\documentclass[dvipdfmx]{article}
\usepackage{amsthm, amscd, amsmath, amsfonts, mathrsfs}
\newtheorem{theo}{Theorem}[section]
\newtheorem{proposition}[theo]{Proposition}
\newtheorem{lemma}[theo]{Lemma}

\newtheorem{rem}[theo]{Remark}
\newtheorem{definition}[theo]{Definition}
\newtheorem{example}[theo]{Example}
\input xy
\xyoption{all}

\begin{document}

\title{Remarks on the Homological Mirror Symmetry for Tori}

\author{Kazushi Kobayashi\footnote{Department of Mathematics, Graduate School of Science, Osaka University, Toyonaka, Osaka, 560-0043, Japan. E-mail : k-kobayashi@cr.math.sci.osaka-u.ac.jp. 2010 Mathematics Subject Classification : 14J33 (primary), 53D18, 14F05 (secondary). Keywords : torus, generalized complex geometry, homological mirror symmetry.}}

\date{}

\maketitle

\maketitle

\begin{abstract}
Let us consider an $n$-dimensional complex torus $T^{2n}_{J=T}:=\mathbb{C}^n/2\pi (\mathbb{Z}^n\oplus T\mathbb{Z}^n)$. Here, $T$ is a complex matrix of order $n$ whose imaginary part is positive definite. In particular, when we consider the case $n=1$, the complexified symplectic form of a mirror partner of $T^2_{J=T}$ is defined by using $-\frac{1}{T}$ or $T$. However, if we assume $n\geq 2$ and that $T$ is a singular matrix, we can not define a mirror partner of $T^{2n}_{J=T}$ as a natural generalization of the case $n=1$ to the higher dimensional case. In this paper, we propose a way to avoid this problem, and discuss the homological mirror symmetry.
\end{abstract}

\tableofcontents

\section{Introduction}
Mirror symmetry is a duality between a complex manifold $M$ and a symplectic manifold $\check{M}$. In 1994, as a categorical formulation of it, Kontsevich proposed a conjecture called the homological mirror symmetry \cite{Kon}. For a given mirror pair $(M,\check{M})$, this conjecture states that there exists an equivalence
\begin{equation*}
D^b(Coh(M))\cong Tr(Fuk(\check{M}))
\end{equation*}
as triangulated categories. Here, $D^b(Coh(M))$ is the bounded derived category of coherent sheaves on $M$, and $Tr(Fuk(\check{M}))$ is the derived category of the Fukaya category $Fuk(M)$ on $\check{M}$ \cite{Fukaya category} which is obtained by the Bondal-Kapranov-Kontsevich construction \cite{bondal}, \cite{Kon}. In particular, a pair of tori is a typical example of a mirror pair, so there are various studies of the homological mirror symmetry for tori (\cite{elliptic}, \cite{Fuk}, \cite{abouzaid} etc.).

When we discuss the mirror symmetry, we first consider a suitable pair consisting of a complex manifold and a symplectic manifold which is called a mirror pair. For a given complex or symplectic manifold, to define a mirror partner of it is a fundamental, and difficult problem. In this paper, in order to address this problem, we focus on a framework called the generalized complex geometry (generalized complex structures, generalized complex manifolds) which is introduced by Hitchin \cite{hitchin} (see also \cite{Gual}). Since generalized complex structures provide a way of treating complex structures and symplectic structures uniformly, we can naturally regard both complex manifolds and symplectic manifolds as generalized complex manifolds. In particular, we can also apply this framework to the study of the mirror symmetry. For example, in \cite{part1}, \cite{part2}, Oren Ben-Bassat introduced the notion of $\nabla$-semi-flat generalized complex structures. Roughly speaking, in \cite{part1}, \cite{part2}, he focuses on a real vector (torus) bundle $V\rightarrow M$ of rank $n$ with a flat connection $\nabla$ and its dual $(V^*\rightarrow M, \nabla^*)$ on the same $n$-dimensional smooth real base manifold $M$, and interprets the mirror symmetry as a duality between the generalized complex structure on the total space of $V$ and the generalized complex structure on the total space of $V^*$ by the T-duality (this construction depends on the choice of $\nabla$ on $V$). In this paper, we consider the case of tori, namely, take the tangent bundle $TT^n$ on an $n$-dimensional real torus $T^n$ with the trivial connection $\nabla$ as $(V\rightarrow M, \nabla)$ in the context of his works \cite{part1}, \cite{part2}. Moreover, for instance, the formulation of the mirror symmetry for tori via the generalized complex geometry is also discussed in \cite{Kaj} (see also \cite{Kim}). 

On the other hand, the SYZ construction \cite{SYZ} proposes a way of constructing mirror pairs geometrically. When we focus on tori, the generalized complex geometry relates this description in the following sense. In general, the group $O(n,n;\mathbb{Z})$ is defined by
\begin{equation*}
O(n,n;\mathbb{Z}):=\left\{ g\in M(2n;\mathbb{Z}) \ | \ g^t J g=J, \ J:=\left( \begin{array}{ccc} O & I_n \\ I_n & O \end{array} \right) \right\},
\end{equation*}
where $I_n$ is the identity matrix of order $n$. Let us consider a $2n$-dimensional real torus $T^{2n}$ equipped with a generalized complex structure $\mathcal{I} : \Gamma (TT^{2n}\oplus T^*T^{2n})\rightarrow \Gamma (TT^{2n}\oplus T^*T^{2n})$, and set
\begin{equation*}
g_{24}:=\left( \begin{array}{cccc} I_n & O & O & O \\ O & O & O & I_n \\ O & O & I_n & O \\ O & I_n & O & O \end{array} \right) \in O(2n,2n;\mathbb{Z}), 
\end{equation*}
where $\Gamma (TT^{2n}\oplus T^*T^{2n})$ denotes the space of smooth sections of $TT^{2n}\oplus T^*T^{2n}$. Then, 
\begin{equation*}
g_{24}^{-1}\mathcal{I}g_{24}
\end{equation*}
again defines a generalized complex structure, and we can treat $(T^{2n}, g_{24}^{-1}\mathcal{I}g_{24})$ as a mirror partner of $(T^{2n}, \mathcal{I})$. Based on the SYZ construction, we regard $(T^{2n}, \mathcal{I})$ and $(T^{2n}, g_{24}^{-1}\mathcal{I}g_{24})$ as the trivial special Lagrangian torus fibrations on the same base space $T^n$, namely, each fiber of these fibrations is also $T^n$. In this situation, the above transformation means that we fix the base space $T^n$ corresponding to the $(1,1)$-block $I_n$ and the $(3,3)$-block $I_n$ of $g_{24}$ and take the T-duality along the remaining $T^n$ regarded as a fiber on the base space $T^n$. In this sense, the group $O(n,n;\mathbb{Z})$ is also called the T-duality group. 

Now, we explain the purpose of this paper. Let $T$ be a complex matrix of order $n$ such that $\mathrm{Im}T$ is positive definite. By using this complex matrix $T$, we can define the $n$-dimensional complex torus
\begin{equation*}
T^{2n}_{J=T}:=\mathbb{C}^n/2\pi (\mathbb{Z}^n\oplus T\mathbb{Z}^n).
\end{equation*}
We would like to define a mirror partner $\check{T}^{2n}_{J=T}$ of this complex torus $T^{2n}_{J=T}$ following the framework of the generalized complex geometry. Let us first consider the case $T=\mathbf{i}T_I$, where $\mathbf{i}=\sqrt{-1}$ and $T_I\in M(n;\mathbb{R})$ is positive definite. Clearly, this $T$ has the inverse $T^{-1}=-\mathbf{i}T_I^{-1}$. Then, according to the discussions in \cite{part1}, \cite{part2} (see also section 3 in \cite{Kaj}), we can define a mirror partner $\check{T}^{2n}_{J=T}$ by employing the complexified symplectic form
\begin{equation}
d\check{x}^t (-T^{-1})^t d\check{y}=d\check{x}^t \mathbf{i}(T_I^{-1})^t d\check{y}, \label{csf}
\end{equation}
where $\check{x}:=(x^1,\cdots, x^n)^t$, $\check{y}:=(y^1,\cdots, y^n)^t$ denote the local coordinates of $\check{T}^{2n}_{J=T}$, and $d\check{x}:=(dx^1,\cdots, dx^n)^t$, $d\check{y}:=(dy^1,\cdots, dy^n)^t$. In particular, the T-duality in this correspondence is described by the matrix $g_{24}\in O(2n,2n;\mathbb{Z})$. Therefore, for general $T$, it is expected that we can define a mirror partner $\check{T}^{2n}_{J=T}$ equipped with the complexified symplectic structure of the form $d\check{x}^t (-T^{-1})^t d\check{y}$ by considering the corresponding B-field transform of (\ref{csf}). Actually, for example, such a B-field transform is already discussed in \cite{Kaj}, and there, it is also proved that the above expectation holds true in the case of elliptic curves (see the relation (41) in \cite{Kaj}). However, if we assume $n\geq 2$, the complex matrix $T$ does not necessarily become a non-singular matrix. Hence, although we can define the complexified symplectic form of $\check{T}^{2n}_{J=T}$ by $d\check{x}^t (-T^{-1})^t d\check{y}$ as a natural generalization of the case of elliptic curves, i.e., $-\frac{1}{T}dx^1\wedge dy^1$ to the higher dimensional case if $\mathrm{det}T\not=0$, we can not define a mirror partner $\check{T}^{2n}_{J=T}$ of $T^{2n}_{J=T}$ by using the similar way if $\mathrm{det}T=0$. On the other hand, when we consider the T-duality, we can also choose
\begin{equation*}
g_{13}:=\left( \begin{array}{cccc} O & O & I_n & O \\ O & I_n & O & O \\ I_n & O & O & O \\ O & O & O & I_n \end{array} \right) \in O(2n,2n;\mathbb{Z})
\end{equation*}
instead of $g_{24}\in O(2n,2n;\mathbb{Z})$. Then, even if $\mathrm{det}T=0$, we can define the complexified symplectic form of $\check{T}^{2n}_{J=T}$ by 
\begin{equation*}
d\check{x}^t T d\check{y}.
\end{equation*}
However, when we discuss the homological mirror symmetry on this mirror pair $(T^{2n}_{J=T}, \check{T}^{2n}_{J=T})$, we need to use the inverse of $T$ in the definition of complex geometric objects of the category which should be defined on $T^{2n}_{J=T}$. Concerning these facts, in this paper, we propose a way to avoid this problem, and discuss the homological mirror symmetry.

This paper is organized as follows. In section 2, we recall the definition of generalized complex manifolds and give several examples of them. Moreover, we also mention the definition of mirror pairs of tori in the framework of the generalized complex geometry. In sections 3, 4, 5, when we treat $T$, we consider the case $\mathrm{det}T=0$ only. In section 3, we propose a way to construct a mirror partner $\check{T}^{2n}_{J=T}$ of $T^{2n}_{J=T}$. In section 4, we consider the homological mirror symmetry setting for the mirror pair $(T^{2n}_{J=T}, \check{T}^{2n}_{J=T})$ which is obtained in section 3. We also provide an analogue of \cite[Theorem 5.1]{kazushi4} on $(T^{2n}_{J=T}, \check{T}^{2n}_{J=T})$, and this result is given in Theorem \ref{bijection2}. In section 5, we investigate holomorphic vector bundles on $T^{2n}_{J=T}$ which are employed in the homological mirror symmetry setting for $(T^{2n}_{J=T}, \check{T}^{2n}_{J=T})$. In particular, since those holomorphic vector bundles are simple projectively flat bundles, we also provide an interpretation of them in the language of factors of automorphy. This result is given in Theorem \ref{maintheorem}, and it is an analogue of \cite[Theorem 3.6]{kazushi2}.

\section{Generalized complex geometry and mirror symmetry for tori}
In this section, we recall the definition of generalized (almost) complex structures briefly, and interpret the mirror symmetry for tori in this framework (\cite{Gual}, \cite{part1}, \cite{part2}, \cite{Kaj}, \cite{Kim} etc.). Throughout this section, for a manifold $M$ and a vector bundle $E$, $TM$, $T^*M$ and $\Gamma (E)$ denote the tangent bundle of $M$, the cotangent bundle of $M$ and the space of smooth sections of $E$, respectively. 

Now, let $M$ be a $2n$-dimensional real smooth manifold. We define a quadratic form 
\begin{equation}
\langle \ \cdot \ , \ \cdot \ \rangle \ : \ \Gamma (TM\oplus T^*M)\otimes \Gamma (TM\oplus T^*M) \longrightarrow C^{\infty }(M) \label{quadratic}
\end{equation}
by $\langle X+\alpha ,Y+\beta \rangle:=\alpha (Y)+\beta (X)$, where $X$, $Y\in \Gamma (TM)$ and $\alpha $, $\beta \in \Gamma (T^*M)$. Then, we can give the definition of generalized almost complex structures (manifolds) as follows.
\begin{definition}[Generalized almost complex structures (manifolds)]
A generalized almost complex structure $\mathcal{I}$ on a smooth manifold $M$ is a linear map $\mathcal{I} : \Gamma (TM\oplus T^*M)\rightarrow \Gamma (TM\oplus T^*M)$ which satisfies $\mathcal{I}^2=-1$ and preserves the quadratic form $(\ref{quadratic})$, i.e., $\langle \mathcal{I}(X+\alpha ), \mathcal{I}(Y+\beta ) \rangle=\langle X+\alpha , Y+\beta \rangle$ $(X$, $Y\in \Gamma (TM)$, $\alpha $, $\beta \in \Gamma (T^*M))$. In particular, a pair $(M,\mathcal{I})$ of a smooth manifold $M$ and a generalized almost complex structure $\mathcal{I}$ is called a generalized almost complex manifold.
\end{definition}
We give several examples of generalized almost complex manifolds. 
\begin{example} \label{ex1}
Let $(M,J)$ be an almost complex manifold, i.e., $J : \Gamma (TM)\rightarrow \Gamma (TM)$ is a linear map such that $J^2=-1$, and we consider the adjoint map $J^* : \Gamma (T^*M)\rightarrow \Gamma (T^*M)$, $(J^*(\alpha ))(X):=\alpha (J(X))$ $(X\in \Gamma (TM), \alpha \in \Gamma (T^*M))$ of $J : \Gamma (TM)\rightarrow \Gamma (TM)$. Then, the map $\mathcal{I}_J := J\oplus (-J^*) : \Gamma (TM\oplus T^*M)\rightarrow \Gamma (TM\oplus T^*M)$ satisfies $\mathcal{I}_J^2=-1$ by $J^2=-1$, and preserves the quadratic form $(\ref{quadratic})$. Hence, almost complex manifolds are examples of generalized almost complex manifolds.
\end{example}
\begin{example} \label{ex2}
Let $\omega $ be a nondegenerate 2-form on $M$. For the linear map $\omega  : \Gamma (TM)\rightarrow \Gamma (T^*M)$, $X\mapsto \omega (X, \cdot )$ $(X\in \Gamma (TM))$ associated to the nondegenerate 2-form $\omega $, we can express the representation matrix $\omega _{rep}$ of $\omega  : \Gamma (TM)\rightarrow \Gamma (T^*M)$ by
\begin{equation*}
\omega _{rep} = \left( \begin{array}{ccc} O & -\omega \\ \omega ^t & O \end{array} \right).
\end{equation*}
Here, note that $\omega _{rep}$ is an alternating matrix. By using this matrix $\omega _{rep}$, we can express the representation matrix of the adjoint map $\omega ^* : \Gamma (T^*M)\rightarrow \Gamma (TM)$, $(\omega (X))(\omega ^*(\alpha ))=\alpha (X)$ $(X\in \Gamma (TM), \alpha \in \Gamma (T^*M))$ of $\omega : \Gamma (TM)\rightarrow \Gamma (T^*M)$ by $-\omega _{rep}^{-1}$. Then, the map
\begin{equation*}
\mathcal{I}_{\omega } := \left( \begin{array}{ccc} O & -\omega _{rep}^{-1} \\ \omega _{rep} & O \end{array} \right) : \Gamma (TM\oplus T^*M)\rightarrow \Gamma (TM\oplus T^*M)
\end{equation*}
satisfies $\mathcal{I}_{\omega }^2 =-1$ by alternativity of $\omega _{rep}$, and preserves the quadratic form $(\ref{quadratic})$. Hence, $(M,\omega )$ forms a generalized almost complex manifold.
\end{example}
\begin{example} \label{ex3}
Let $(M,\mathcal{I})$ be a generalized almost complex manifold. In local matrix expression, we consider the transformation on $\Gamma (TM\oplus T^*M)$ which is defined by 
\begin{equation*}
\left( \begin{array}{ccc} I_{2n} & O \\ B & I_{2n} \end{array} \right).
\end{equation*}
Here, $B$ is a matrix of order $2n$ and $I_{2n}$ denotes the identity matrix of order $2n$. This means that the matrix $B$ defines a 2-form $B\in \Gamma (\wedge ^2 T^*M)$. Here, we define the linear map $\mathcal{I}(B) : \Gamma (TM\oplus T^*M)\rightarrow \Gamma (TM\oplus T^*M)$ by
\begin{equation*}
\mathcal{I}(B) := \left( \begin{array}{ccc} I_{2n} & O \\ B & I_{2n} \end{array} \right) \mathcal{I} \left( \begin{array}{ccc} I_{2n} & O \\ B & I_{2n} \end{array} \right)^{-1} = \left( \begin{array}{ccc} I_{2n} & O \\ B & I_{2n} \end{array} \right) \mathcal{I} \left( \begin{array}{ccc} I_{2n} & O \\ -B & I_{2n} \end{array} \right).
\end{equation*}
Then, $(M, \mathcal{I}(B))$ also forms a generalized almost complex manifold. In this case, $\mathcal{I}(B)$ is called a B-field transform of $\mathcal{I}$.
\end{example}
Strictly speaking, in order to define generalized complex structures (manifolds), we need to discuss integrability of generalized almost complex structures (manifolds). Although we omit the details of discussions of integrability conditions here, we remark some facts below.
\begin{rem}
In Example \ref{ex1}, the condition that $J$ is integrable is equivalent to that $\mathcal{I}_J$ is integrable. Thus, complex manifolds are examples of generalized complex manifolds.
\end{rem}
\begin{rem}
In example \ref{ex2}, $\omega $ is a closed 2-form if and only if $\mathcal{I}_{\omega }$ is integrable. Thus, symplectic manifolds are examples of generalized complex manifolds.
\end{rem}
\begin{rem}
In example \ref{ex3}, the 2-form $B$ is a closed 2-form if and only if $\mathcal{I}(B)$ is integrable. 
\end{rem}
Hereafter, for a given complex torus $\mathbb{C}^n/2\pi (\mathbb{Z}^n\oplus T\mathbb{Z}^n)$, sometimes we denote $T^{2n}_{J=T}$ instead of $\mathbb{C}^n/2\pi (\mathbb{Z}^n\oplus T\mathbb{Z}^n)$ :
\begin{equation*}
T^{2n}_{J=T}:=\mathbb{C}^n/2\pi (\mathbb{Z}^n\oplus T\mathbb{Z}^n).
\end{equation*}
Similarly, for a given complexified symplectic torus $(T^{2n},\tilde{\omega } =dp^t \tau dq)$, where $p:=(p^1,\cdots, p^n)^t$, $q:=(q^1,\cdots,q^n)^t$ are the local coordinates of $T^{2n}$ and $dp:=(dp^1,\cdots,dp^n)^t$, $dq:=(dq^1,\cdots,dq^n)^t$, sometimes we denote $T^{2n}_{\tilde{\omega } =\tau }$ instead of $(T^{2n},\tilde{\omega } =dp^t \tau dq)$ :
\begin{equation*}
T^{2n}_{\tilde{\omega}=\tau}:=(T^{2n},\tilde{\omega } =dp^t \tau dq).
\end{equation*}

We give an interpretation of the mirror symmetry for tori in terms of the generalized complex geometry (cf. \cite{part1}, \cite{part2}, \cite{Kaj}). In general, there are two ways how to define a mirror partner of a $2n$-dimensional real torus $T^{2n}$ equipped with a generalized complex structure $\mathcal{I}$ (Definition \ref{mirror1} and Definition \ref{mirror2} correspond to the T-duality by $g_{24}\in O(2n,2n;\mathbb{Z})$ and the T-duality by $g_{13}\in O(2n,2n;\mathbb{Z})$ in the introduction, respectively).
\begin{definition} \label{mirror1}
For a $2n$-dimensional real torus $(T^{2n}, \mathcal{I})$, where $\mathcal{I}$ is a generalized complex structure on $T^{2n}$, its mirror transform $(T^{2n}, \check{\mathcal{I}})$ is defined by
\begin{equation*}
\check{\mathcal{I}}:=\left( \begin{array}{ccccc} I_n & O & O & O \\ O & O & O & I_n \\ O & O & I_n & O \\ O & I_n & O & O \end{array} \right) \mathcal{I} \left( \begin{array}{ccccc} I_n & O & O & O \\ O & O & O & I_n \\ O & O & I_n & O \\ O & I_n & O & O \end{array} \right).
\end{equation*}
\end{definition}
\begin{definition} \label{mirror2}
For a $2n$-dimensional real torus $(T^{2n}, \mathcal{I})$, where $\mathcal{I}$ is a generalized complex structure on $T^{2n}$, its mirror transform $(T^{2n}, \check{\mathcal{I}})$ is defined by
\begin{equation*}
\check{\mathcal{I}}:=\left( \begin{array}{ccccc} O & O & I_n & O \\ O & I_n & O & O \\ I_n & O & O & O \\ O & O & O & I_n \end{array} \right) \mathcal{I} \left( \begin{array}{ccccc} O & O & I_n & O \\ O & I_n & O & O \\ I_n & O & O & O \\ O & O & O & I_n \end{array} \right).
\end{equation*}
\end{definition}
We first observe a mirror partner of a given complex torus $T^{2n}_{J=T}$ following Definition \ref{mirror1}. When we denote $T=T_R+\mathbf{i}T_I$ with $T_R:=$Re$T$, $T_I:=$Im$T$ and $\mathbf{i}=\sqrt{-1}$, the corresponding generalized complex structure $\mathcal{I}_J$ is expressed as
\begin{equation} \label{I_J}
\left( \begin{array}{cccc} -T_R T_I^{-1} & -T_I -T_R T_I^{-1} T_R & O & O \\ T_I^{-1} & T_I^{-1} T_R & O & O \\ O & O & (T_I^{-1})^t T_R^t & -(T_I^{-1})^t \\ O & O & T_I^t + T_R^t(T_I^{-1})^t T_R^t & -T_R^t(T_I^{-1})^t \end{array} \right).
\end{equation}
On the other hand, for a complexified symplectic torus $T^{2n}_{\tilde{\omega } =B+\mathbf{i}\omega }$ and two matrices
\begin{equation*}
\omega _{rep} := \left( \begin{array}{ccc} O & -\omega \\ \omega ^t & O \end{array} \right), \ \tilde{B} := \left( \begin{array}{ccc} O & -B \\ B^t & O \end{array} \right),
\end{equation*}
the corresponding generalized complex structure 
\begin{equation*}
\mathcal{I}_{\omega }(B) = \left( \begin{array}{ccc} I_{2n} & O \\ \tilde{B} & I_{2n} \end{array} \right) \left( \begin{array}{ccc} O & -\omega _{rep}^{-1} \\ \omega _{rep} & O \end{array} \right) \left( \begin{array}{ccc} I_{2n} & O \\ -\tilde{B} & I_{2n} \end{array} \right)
\end{equation*}
is expressed as
\begin{equation*} 
\left( \begin{array}{cccc} (\omega ^{-1})^t B^t & O & O & -(\omega ^{-1})^t \\ O & \omega ^{-1} B & \omega ^{-1} & O \\ O & -\omega -B \omega ^{-1} B & -B \omega ^{-1} & O \\ \omega ^t + B^t (\omega ^{-1})^t B^t & O & O & -B^t (\omega ^{-1} )^t \end{array} \right),
\end{equation*}
so by Definition \ref{mirror1}, its mirror dual $\check{\mathcal{I}_{\omega }}(B)$ is 
\begin{equation} \label{I_omega}
\left( \begin{array}{cccc} (\omega ^{-1})^t B^t & -(\omega ^{-1})^t & O & O \\ \omega ^t + B^t (\omega ^{-1})^t B^t & -B^t (\omega ^{-1})^t & O & O \\ O & O & -B\omega ^{-1} & -\omega -B \omega ^{-1} B \\ O & O & \omega ^{-1} & \omega ^{-1} B \end{array} \right).
\end{equation}
Hence, by comparing the expression (\ref{I_J}) of $\mathcal{I_J}$ and the expression (\ref{I_omega}) of $\check{\mathcal{I}_{\omega }}(B)$, we obtain the following equalities :
\begin{align}
&-T_R T_I^{-1} = (\omega ^{-1})^t B^t, \label{eq1} \\
&-T_I -T_R T_I^{-1}T_R = -(\omega ^{-1})^t, \label{eq2} \\
&T_I^{-1} = \omega ^t + B^t (\omega ^{-1})^t B^t, \label{eq3} \\
&T_I^{-1} T_R = -B^t (\omega ^{-1})^t. \label{eq4} 
\end{align}
We obtain the two relations 
\begin{align}
&-T_R^t \omega -T_I^t B = O, \label{eq5} \\
&T_R^t B = -T_R^t(T_I^{-1})^t T_R^t \omega \label{eq6}
\end{align}
by deforming the equality (\ref{eq1}). In particular, by using the equality (\ref{eq6}), one has
\begin{align*}
T_I^t\omega -T_R^t B &=T_I^t\omega +T_R^t (T_I^{-1})^t T_R^t \omega \\
&=(T_I^t + T_R^t (T_I^{-1})^t T_R^t )\omega .
\end{align*}
On the other hand, the equality (\ref{eq2}) is equivalent to 
\begin{equation*}
T_I^t + T_R^t (T_I^{-1})^t T_R^t =\omega ^{-1},
\end{equation*}
so we obtain the following :
\begin{equation}
T_I^t\omega -T_R^t B = I_n. \label{eq7}
\end{equation}
Thus, by using the equalities (\ref{eq5}), (\ref{eq7}), we see
\begin{align}
-T^t(B+\mathbf{i}\omega )&=-(T_R+\mathbf{i}T_I)^t(B+\mathbf{i}\omega ) \notag \\ 
&=T_I^t\omega -T_R^t B + \mathbf{i}(-T_R^t\omega -T_I^t B) \notag \\ 
&=I_n. \label{right}
\end{align}
Similarly, by using the equalities (\ref{eq2}), (\ref{eq4}), the relation 
\begin{equation}
(B+\mathbf{i}\omega )(-T^t)=I_n \label{left}
\end{equation}
also holds. Clearly, the two relations (\ref{right}), (\ref{left}) state $B+\mathbf{i}\omega =-(T^{-1})^t$, and this fact indicates that the matrix $T$ is invertible. Furthermore, when we set 
\begin{equation*}
B=\mathrm{Re}(-(T^{-1})^t),\ \omega =\mathrm{Im}(-(T^{-1})^t),
\end{equation*}
the equality (\ref{eq3}) is also satisfied automatically. However, if we assume $n\geq 2$, the matrix $T$ does not necessarily become a non-singular matrix. For example,
\begin{equation*}
T:=\left( \begin{array}{ccc} \mathbf{i} & 1 \\ -1 & \mathbf{i} \end{array} \right) = \left( \begin{array}{ccc} 0 & 1 \\ -1 & 0 \end{array} \right) + \mathbf{i}\cdot I_2.
\end{equation*}
Therefore, if the matrix $T$ is a singular matrix, we can not define a mirror partner of the complex torus $T^{2n}_{J=T}$ following Definition \ref{mirror1}. 

We also mention the construction of mirror partners of complex tori following Definition \ref{mirror2}. In this case, for a complex torus $T^{2n}_{J=T}$, we can give the complexified symplectic structure of the mirror partner of the complex torus $T^{2n}_{J=T}$ by using the matrix $T$, namely, $\check{T}^{2n}_{J=T}:=T^{2n}_{\tilde{\omega }=T}$. Therefore, we can define a mirror partner of the complex torus $T^{2n}_{J=T}$ even if $T$ is a singular matrix. However, when we discuss the homological mirror symmetry on this mirror pair $(\check{T}^{2n}_{J=T}, T^{2n}_{J=T})$, we need to use the inverse of $T$ in the definition of objects of the category which should be defined on $T^{2n}_{J=T}$ (we also mention the details of this problem in Remark \ref{Tinverse}).

Thus, in any case, we need to use the inverse of $T$. We propose a way to avoid this problem in sections 3, 4. 

\section{On mirror partners of complex tori}
As discussed in section 2, for a complex torus $T^{2n}_{J=T}$, we should consider the case such that the matrix $T$ is a singular matrix. Therefore, we choose an arbitrary singular matrix $T\in M(n;\mathbb{C})$ whose imaginary part is positive definite, and fix it. Note that we use the same matrix $T$ in sections 4, 5. The purpose of this section is to construct a mirror partner of the complex torus $T^{2n}_{J=T}$ which is defined by the fixed singular matrix $T$. 

First, as the preparation of the main discussions, we prove the following lemma.
\begin{lemma} \label{delta}
For the singular matrix $T\in M(n;\mathbb{C})$, there exists a matrix $\delta \in M(n;\mathbb{Z})$ such that $T-\delta $ is a non-singular matrix.
\end{lemma}
\begin{proof}
Now, although $T$ is a singular matrix, since $\mathrm{Im}T$ is positive definite, $\mathrm{rank}T\not=0$. Therefore, we may assume 
\begin{equation*}
\mathrm{rank}T=m, \ m=1,\cdots, n-1,
\end{equation*}
and it is enough to consider the case
\begin{equation*}
T=\left( \begin{array}{@{\,}ccccccc@{\,}} 
t_{11} & \ldots & t_{1 i_1} & \ldots & t_{1 i_m} & \ldots & t_{1n} \\
\vdots & \ddots & \vdots & \ddots & \vdots & \ddots & \vdots \\
t_{m1} & \ldots & t_{m i_1} & \ldots & t_{m i_m} & \ldots & t_{mn} \\
c_{m+1}t_{j_{m+1}1} & \ldots & c_{m+1}t_{j_{m+1}i_1} & \ldots & c_{m+1}t_{j_{m+1}i_m} & \ldots & c_{m+1}t_{j_{m+1}n} \\
\vdots & \ddots & \vdots & \ddots & \vdots & \ddots & \vdots \\
c_{n}t_{j_{n}1} & \ldots & c_{n}t_{j_{n}i_1} & \ldots & c_{n}t_{j_{n}i_m} & \ldots & c_{n}t_{j_{n}n} 
\end{array} \right),
\end{equation*}
where 
\begin{align*}
&1\leq i_1<\cdots<i_m\leq n, \\
&j_k=1,\cdots, m \ (k=m+1,\cdots, n), \\
&c_k \in \mathbb{C}^{\times} \ (k=m+1,\cdots, n), \\
&\mathrm{det}\left( \begin{array}{ccccc} 
t_{1i_1} & \ldots & t_{1i_m} \\
\vdots & \ddots & \vdots \\
t_{mi_1} & \ldots & t_{mi_m} 
\end{array} \right) \not= 0.
\end{align*}
For this $T$, for example, we may define $\delta=(\delta_{ij})\in M(n;\mathbb{Z})$ by
\begin{align*}
&\delta_{(m+1)(i_m+1)}=\cdots=\delta_{(m+n-i_m)(n)}=1, \\
&\delta_{(m+n-i_m+1)(1)}=\cdots=\delta_{(m+n-i_m+i_1-1)(i_1-1)}=1, \\
&\delta_{(m+n-i_m+i_1)(i_1+1)}=\cdots=\delta_{(m+n-i_m+i_2-2)(i_2-1)}=1, \\
&\hspace{40mm} \vdots \\
&\delta_{(n-i_m+i_{m-1}+2)(i_{m-1}+1)}=\cdots=\delta_{(n)(i_m-1)}=1,
\end{align*}
and all other components are $0$. Namely, in the definition of this $\delta \in M(n;\mathbb{Z})$,
\begin{equation*}
(n-i_m)+(i_1-1)+(i_2-i_1-1)+\cdots+(i_m-i_{m-1}-1)=n-m
\end{equation*}
components are $1$. Then, by the multilinearity of determinants, we see
\begin{equation*}
\mathrm{det}(T-\delta)=\pm \mathrm{det}\left( \begin{array}{ccccc} 
t_{1i_1} & \ldots & t_{1i_m} \\
\vdots & \ddots & \vdots \\
t_{mi_1} & \ldots & t_{mi_m} 
\end{array} \right) \not= 0.
\end{equation*}
Although this completes the proof itself, finally, we give an example. Let us consider the case such that $n=5$, $m=2$, $i_1=1$, $i_2=3$, $j_3=1$, $j_4=1$, $j_5=2$, i.e.,
\begin{equation*}
T=\left( \begin{array}{ccccc} 
t_{11} & t_{12} & t_{13} & t_{14} & t_{15} \\
t_{21} & t_{22} & t_{23} & t_{24} & t_{25} \\
c_1t_{11} & c_1t_{12} & c_1t_{13} & c_1t_{14} & c_1t_{15} \\
c_2t_{11} & c_2t_{12} & c_2t_{13} & c_2t_{14} & c_2t_{15} \\
c_3t_{21} & c_3t_{22} & c_3t_{23} & c_3t_{24} & c_3t_{25} 
\end{array} \right),
\end{equation*}
where $c_1$, $c_2$, $c_3 \in \mathbb{C}^{\times}$ and 
\begin{equation*}
\mathrm{det}\left( \begin{array}{ccc} t_{11} & t_{13} \\ t_{21} & t_{23} \end{array} \right)\not=0.
\end{equation*}
Then, according to the above definition, we can give $\delta \in M(5;\mathbb{Z})$ by
\begin{equation*}
\delta_{34}=\delta_{45}=1, \ \delta_{52}=1,
\end{equation*}
and all other components are $0$, i.e.,
\begin{equation*}
\delta=\left( \begin{array}{ccccc}
0 & 0 & 0 & 0 & 0 \\
0 & 0 & 0 & 0 & 0 \\
0 & 0 & 0 & 1 & 0 \\
0 & 0 & 0 & 0 & 1 \\
0 & 1 & 0 & 0 & 0 
\end{array} \right).
\end{equation*}
In this situation, by a direct calculation, we see
\begin{align*}
\mathrm{det}(T-\delta)&=\mathrm{det}T+\mathrm{det}\left( \begin{array}{ccccc} 
t_{11} & t_{12} & t_{13} & t_{14} & t_{15} \\
t_{21} & t_{22} & t_{23} & t_{24} & t_{25} \\
c_1t_{11} & c_1t_{12} & c_1t_{13} & c_1t_{14} & c_1t_{15} \\
c_2t_{11} & c_2t_{12} & c_2t_{13} & c_2t_{14} & c_2t_{15} \\
0 & -1 & 0 & 0 & 0
\end{array} \right) \\
&\hspace{5mm} +\mathrm{det}\left( \begin{array}{ccccc} 
t_{11} & t_{12} & t_{13} & t_{14} & t_{15} \\
t_{21} & t_{22} & t_{23} & t_{24} & t_{25} \\
c_1t_{11} & c_1t_{12} & c_1t_{13} & c_1t_{14} & c_1t_{15} \\
0 & 0 & 0 & 0 & -1 \\
c_3t_{21} & c_3t_{22} & c_3t_{23} & c_3t_{24} & c_3t_{25} 
\end{array} \right) \\
&\hspace{5mm} +\mathrm{det}\left( \begin{array}{ccccc} 
t_{11} & t_{12} & t_{13} & t_{14} & t_{15} \\
t_{21} & t_{22} & t_{23} & t_{24} & t_{25} \\
c_1t_{11} & c_1t_{12} & c_1t_{13} & c_1t_{14} & c_1t_{15} \\
0 & 0 & 0 & 0 & -1 \\
0 & -1 & 0 & 0 & 0
\end{array} \right) \\
&\hspace{5mm} +\mathrm{det}\left( \begin{array}{ccccc} 
t_{11} & t_{12} & t_{13} & t_{14} & t_{15} \\
t_{21} & t_{22} & t_{23} & t_{24} & t_{25} \\
0 & 0 & 0 & -1 & 0 \\
c_2t_{11} & c_2t_{12} & c_2t_{13} & c_2t_{14} & c_2t_{15} \\
c_3t_{21} & c_3t_{22} & c_3t_{23} & c_3t_{24} & c_3t_{25} 
\end{array} \right) 
\end{align*}
\begin{align*}
&\hspace{14mm} +\mathrm{det}\left( \begin{array}{ccccc} 
t_{11} & t_{12} & t_{13} & t_{14} & t_{15} \\
t_{21} & t_{22} & t_{23} & t_{24} & t_{25} \\
0 & 0 & 0 & -1 & 0 \\
c_2t_{11} & c_2t_{12} & c_2t_{13} & c_2t_{14} & c_2t_{15} \\
0 & -1 & 0 & 0 & 0
\end{array} \right) \\
&\hspace{14mm} +\mathrm{det}\left( \begin{array}{ccccc} 
t_{11} & t_{12} & t_{13} & t_{14} & t_{15} \\
t_{21} & t_{22} & t_{23} & t_{24} & t_{25} \\
0 & 0 & 0 & -1 & 0 \\
0 & 0 & 0 & 0 & -1 \\
c_3t_{21} & c_3t_{22} & c_3t_{23} & c_3t_{24} & c_3t_{25} 
\end{array} \right) \\
&\hspace{14mm} +\mathrm{det}\left( \begin{array}{ccccc} 
t_{11} & t_{12} & t_{13} & t_{14} & t_{15} \\
t_{21} & t_{22} & t_{23} & t_{24} & t_{25} \\
0 & 0 & 0 & -1 & 0 \\
0 & 0 & 0 & 0 & -1 \\
0 & -1 & 0 & 0 & 0
\end{array} \right) \\
&\hspace{11mm}=\mathrm{det}\left( \begin{array}{ccccc} 
t_{11} & t_{12} & t_{13} & t_{14} & t_{15} \\
t_{21} & t_{22} & t_{23} & t_{24} & t_{25} \\
0 & 0 & 0 & -1 & 0 \\
0 & 0 & 0 & 0 & -1 \\
0 & -1 & 0 & 0 & 0
\end{array} \right) \\
&\hspace{11mm}=\mathrm{det}\left( \begin{array}{ccc} t_{11} & t_{13} \\ t_{21} & t_{23} \end{array} \right) \not=0.
\end{align*}
\end{proof}
In general, how to choose a matrix $\delta $ in Lemma \ref{delta} is not unique. We explain the nonuniqueness of the choice of a matrix $\delta $ in Lemma \ref{delta} and related discussions in section 4 (cf. Proposition \ref{derived}).

Now, we consider how to define a mirror partner of the complex torus $T^{2n}_{J=T}$. As described in section 2, we can consider the complexified symplectic torus $T^{2n}_{\tilde{\omega } =T}$ as a mirror partner of the complex torus $T^{2n}_{J=T}$ (see also Definition \ref{mirror2}). On the other hand, since we can choose a matrix $\delta \in M(n;\mathbb{Z})$ such that $T-\delta $ is a non-singular matrix by Lemma \ref{delta}, we fix such a matrix $\delta $. By using this matrix $\delta $, we define the matrix
\begin{equation*}
\mathcal{D}:=\left( \begin{array}{ccc} O & \delta \\ -\delta ^t & O \end{array} \right)\in M(2n;\mathbb{Z}),
\end{equation*}
and transform the generalized complex structure $\mathcal{I}_{\tilde{\omega } =T}=\mathcal{I}_{\mathrm{Im}T}(\mathrm{Re}T)$ of the complexified symplectic torus $T^{2n}_{\tilde{\omega } =T}$ as follows.
\begin{align}
&\left( \begin{array}{ccc} I_{2n} & O \\ \mathcal{D} & I_{2n} \end{array} \right)\mathcal{I}_{\tilde{\omega } =T} \left( \begin{array}{ccc} I_{2n} & O \\ -\mathcal{D} & I_{2n} \end{array} \right) \notag \\
&=\left( \begin{array}{ccc} I_{2n} & O \\ \mathcal{D} & I_{2n} \end{array} \right) \Biggl\{ \left( \begin{array}{ccc} I_{2n} & O \\ \tilde{B} & I_{2n} \end{array} \right) \mathcal{I}_{\mathrm{Im}T} \left( \begin{array}{ccc} I_{2n} & O \\ -\tilde{B} & I_{2n} \end{array} \right) \Biggr\} \left( \begin{array}{ccc} I_{2n} & O \\ -\mathcal{D} & I_{2n} \end{array} \right). \label{Dtransform}
\end{align}
Here, 
\begin{equation*}
\tilde{B} =\left( \begin{array}{ccc} O & -\mathrm{Re}T \\ (\mathrm{Re}T)^t & O \end{array} \right).
\end{equation*}
In particular, we can regard the formula (\ref{Dtransform}) as 
\begin{equation*}
\left( \begin{array}{ccc} I_{2n} & O \\ \tilde{B} +\mathcal{D} & I_{2n} \end{array} \right) \mathcal{I}_{\mathrm{Im}T} \left( \begin{array}{ccc} I_{2n} & O \\ -(\tilde{B} +\mathcal{D}) & I_{2n} \end{array} \right),
\end{equation*}
so by this transform, the symplectic structure $\mathrm{Im}T$ of $T^{2n}_{\tilde{\omega } =T}$ is preserved and the B-field $\mathrm{Re}T$ of $T^{2n}_{\tilde{\omega } =T}$ is only transformed. In fact, the complexified symplectic form corresponding to the formula (\ref{Dtransform}) is given by the non-singular matrix
\begin{equation*}
T-\delta =(\mathrm{Re}T-\delta )+\mathbf{i}\mathrm{Im}T.
\end{equation*}
Concerning the above discussions, \textbf{we propose treating the complexfied symplectic torus} 
\begin{equation*}
T^{2n}_{\tilde{\omega } =T-\delta }
\end{equation*}
\textbf{as a mirror partner of the complex torus} 
\begin{equation*}
T^{2n}_{J=T}. 
\end{equation*}
We discuss the homological mirror symmetry on this mirror pair $(\check{T}^{2n}_{J=T} :=T^{2n}_{\tilde{\omega } =T-\delta }, T^{2n}_{J=T})$ in section 4.

\section{The homological mirror symmetry setting for $(\check{T}^{2n}_{J=T},T^{2n}_{J=T})$}
In this section, we define a class of holomorphic vector bundles on $T^{2n}_{J=T}$ and their mirror dual objects on $\check{T}^{2n}_{J=T}$. More precisely, we first construct a biholomorphic map $\varphi : T^{2n}_{J=T} \stackrel{\sim }{\rightarrow } T^{2n}_{J=T'}$, where $T'$ is a non-singular matrix, and consider a class of holomorphic vector bundles $E\rightarrow T^{2n}_{J=T'}$ based on the SYZ construction. We then employ holomorphic vector bundles $\varphi^*E \rightarrow T^{2n}_{J=T}$ as objects in the complex geometry side, and define their mirror dual objects on $\check{T}^{2n}_{J=T}$ (Theorem \ref{bijection2}).

Let us denote the local complex coordinates of $T^{2n}_{J=T}=\mathbb{C}^n/2\pi (\mathbb{Z}^n\oplus T\mathbb{Z}^n)$ by $z=x+Ty$, where 
\begin{equation*}
z:=(z_1,\cdots, z_n)^t,\ x:=(x_1,\cdots, x_n)^t,\ y:=(y_1,\cdots, y_n)^t.
\end{equation*}
In general, for two complex tori $T^{2n}_{J=T}$, $T^{2n}_{J=T'}$, $T^{2n}_{J=T}$ and $T^{2n}_{J=T'}$ are biholomorphic if and only if $T'=(T\mathscr{C}+\mathscr{A})^{-1}(T\mathscr{D}+\mathscr{B})$, where $\mathscr{A}$, $\mathscr{B}$, $\mathscr{C}$, $\mathscr{D}\in M(n;\mathbb{Z})$ and 
\begin{equation*}
\left( \begin{array}{ccc} \mathscr{A} & \mathscr{B} \\ \mathscr{C} & \mathscr{D} \end{array} \right)\in GL(2n;\mathbb{Z}).
\end{equation*}
Here, we set
\begin{equation*}
\left( \begin{array}{ccc} \mathscr{A} & \mathscr{B} \\ \mathscr{C} & \mathscr{D} \end{array} \right) = \left( \begin{array}{ccc} \delta & I_n \\ -I_n & O \end{array} \right)\in SL(2n;\mathbb{Z})\subset GL(2n;\mathbb{Z}),
\end{equation*}
namely,
\begin{equation*}
T'=(-T+\delta )^{-1}.
\end{equation*}
Let us denote the local complex coordinates of $T^{2n}_{J=T'}=T^{2n}_{J=(-T+\delta )^{-1}}=\mathbb{C}^n/2\pi (\mathbb{Z}^n\oplus (-T+\delta )^{-1}\mathbb{Z}^n)$ by $Z=X+T'Y=X+(-T+\delta )^{-1}Y$, where
\begin{equation*}
Z:=(Z_1,\cdots, Z_n)^t,\ X:=(X_1,\cdots, X_n)^t,\ Y:=(Y_1,\cdots, Y_n)^t.
\end{equation*}
Then, a biholomorphic map $\varphi : T^{2n}_{J=T} \stackrel{\sim }{\rightarrow } T^{2n}_{J=T'}$ is given by
\begin{equation*}
\varphi (z)=(-T+\delta )^{-1} z.
\end{equation*}
When we regard complex manifolds $T^{2n}_{J=T}$ and $T^{2n}_{J=T'}$ as real differentiable manifolds $\mathbb{R}^{2n}/2\pi \mathbb{Z}^{2n}$, the biholomorphic map $\varphi $ is regarded as the diffeomorphism $\varphi : \mathbb{R}^{2n}/2\pi \mathbb{Z}^{2n} \stackrel{\sim }{\rightarrow } \mathbb{R}^{2n}/2\pi \mathbb{Z}^{2n}$ such that
\begin{equation*}
\varphi \left( \begin{array}{ccc} x \\ y \end{array} \right) = \left( \begin{array}{ccc} O & -I_n \\ I_n & \delta \end{array} \right) \left( \begin{array}{ccc} x \\ y \end{array} \right).
\end{equation*}

We consider the complexified symplectic torus $T^{2n}_{\tilde{\omega } =-(T'^{-1})^t}$ as a mirror partner of the complex torus $T^{2n}_{J=T'}$, i.e., $\check{T}^{2n}_{J=T'}:=T^{2n}_{\tilde{\omega } =-(T'^{-1})^t}$ (see also Definition \ref{mirror1}). We denote the local coordinates of $\check{T}^{2n}_{J=T'}$ by $(X^1,\cdots X^n,Y^1,\cdots,Y^n)^t$, and define
\begin{equation*}
\check{X}:=(X^1,\cdots, X^n)^t,\ \check{Y}:=(Y^1,\cdots, Y^n).
\end{equation*}

Now, we explain the homological mirror symmetry setting for $(\check{T}^{2n}_{J=T'},T^{2n}_{J=T'})$ following \cite{kazushi4}. This description is also based on the SYZ construction \cite{SYZ} (cf. \cite{leung}, \cite{A-P}). We assume $r\in \mathbb{N}$, $A=(a_{ij})\in M(n;\mathbb{Z})$ and $\mu :=(\mu _1,\cdots, \mu _n)^t \in \mathbb{C}^n$. Sometimes we denote $\mu =p+T'^tq$ with $p:=(p_1,\cdots, p_n)^t\in \mathbb{R}^n$, $q:=(q_1,\cdots, q_n)^t\in \mathbb{R}^n$. 

First, we explain the complex geometry side, namely, define a class of holomorphic vector bundles
\begin{equation*}
E_{(r,A,\mu, \mathcal{U})} \rightarrow T^{2n}_{J=T'}. 
\end{equation*}
We first construct it as a complex vector bundle, and then discuss when it becomes a holomorphic vector bundle later in Proposition \ref{holomorphic1}. We define $r'\in \mathbb{N}$ by using a given pair $(r,A)\in \mathbb{N}\times M(n;\mathbb{Z})$ as follows. By the theory of elementary divisors, there exist two matrices $\mathcal{A}$, $\mathcal{B}\in GL(n;\mathbb{Z})$ such that 
\begin{equation}
\mathcal{A}A\mathcal{B}=\left( \begin{array}{cccccc} \tilde{a_1} & & & & & \\ & \ddots & & & & \\ & & \tilde{a_s} & & & \\ & & & 0 & & \\ & & & & \ddots & \\ & & & & & 0 \end{array} \right), \label{matAB}
\end{equation}
where $\tilde{a_i}\in \mathbb{N}$ ($i=1,\cdots, s, 1\leq s\leq n$) and $\tilde{a_i}|\tilde{a_{i+1}}$ ($i=1,\cdots, s-1$). Then, we define $r_i'\in \mathbb{N}$ and $a_i'\in \mathbb{Z}$ ($i=1,\cdots, s$) by
\begin{equation*}
\frac{\tilde{a_i}}{r}=\frac{a_i'}{r_i'}, \ gcd(r_i',a_i')=1,
\end{equation*}
where $gcd(m,n)>0$ denotes the greatest common divisor of $m$, $n\in \mathbb{Z}$. By using these, we set
\begin{equation}
r':=r_1'\cdots r_s'\in \mathbb{N}. \label{r'}
\end{equation}
This $r'\in \mathbb{N}$ is uniquely defined by a given pair $(r,A)\in \mathbb{N}\times M(n;\mathbb{Z})$, and it is actually the rank of $E_{(r,A,\mu, \mathcal{U})}$ (in this sense, although we should also emphasize $r'\in \mathbb{N}$ when we denote $E_{(r,A,\mu, \mathcal{U})}$, for simplicity, we use the notation $E_{(r,A,\mu, \mathcal{U})}$ in this paper). In order to define holomorphic vector bundles $E_{(r,A,\mu, \mathcal{U})}$, of course, we must define transition functions of them. However, since the notations of transition functions of $E_{(r,A,\mu, \mathcal{U})}$ are complicated, we only give the rough definition of transition functions of $E_{(r,A,\mu, \mathcal{U})}$, here (details of these discussions are described in subsection 3.1 of \cite{kazushi4}). Let 
\begin{equation*}
s(X_1,\cdots, X_n,Y_1,\cdots, Y_n)
\end{equation*}
be a smooth section of $E_{(r,A,\mu ,\mathcal{U})}$. Then, we define the transition functions of $E_{(r,A,\mu ,\mathcal{U})}$ as follows, where $V_j\in U(r')$, $U_k\in U(r')$ $(j,k=1,\cdots,n)$ and $a_j:=(a_{1j},\cdots,a_{nj})$.
\begin{align*}
&s(X_1,\cdots, X_j+2\pi ,\cdots, X_n, Y_1,\cdots, Y_n)=e^{\frac{\mathbf{i}}{r}a_j Y}V_j \cdot s(X_1,\cdots, X_n, Y_1,\cdots, Y_n), \\
&s(X_1,\cdots, X_n, Y_1,\cdots, Y_k+2\pi , \cdots, Y_n)=U_k \cdot s(X_1,\cdots, X_n, Y_1,\cdots, Y_n).
\end{align*}
In particular, the cocycle condition is expressed as 
\begin{equation*}
V_j V_k =V_k V_j, \ U_j U_k =U_k U_j,\ \zeta ^{-a_{kj}}U_k V_j =V_j U_k,
\end{equation*}
where $\zeta $ is the $r$-th root of 1. We define the set $\mathcal{U}$ of unitary matrices by
\begin{align*}
\mathcal{U}:=\Bigl\{ V_j,\ U_k \in U(r') \ | \ &V_j V_k =V_k V_j, \ U_j U_k =U_k U_j,\ \zeta ^{-a_{kj}}U_k V_j =V_j U_k, \\
&j,k=1,\cdots, n \Bigr\}.
\end{align*}
Of course, how to define the set $\mathcal{U}$ relates closely to (in)decomposability of $E_{(r,A,\mu, \mathcal{U})}$. Here, although we only treat the set $\mathcal{U}$ such that $E_{(r,A,\mu, \mathcal{U})}$ is simple, actually, it is known that the following proposition holds (see \cite[Proposition 3.2]{kazushi4}).
\begin{proposition}
For each quadruple $(r,A,p,q)\in \mathbb{N} \times M(n;\mathbb{Z}) \times \mathbb{R}^n \times \mathbb{R}^n$, we can take a set $\mathcal{U} \not= \emptyset$ such that $E_{(r,A,\mu, \mathcal{U})}$ is simple.
\end{proposition}
Furthermore, we define a connection $\nabla_{(r,A,\mu, \mathcal{U})}$ on $E_{(r,A,\mu, \mathcal{U})}$ locally as 
\begin{equation*}
\nabla_{(r,A,\mu, \mathcal{U})}=d+\omega_{(r,A,\mu, \mathcal{U})}:=d-\frac{\mathbf{i}}{2\pi } \left(  \frac{1}{r}X^t A^t +\frac{1}{r}\mu^t \right) dY\cdot I_{r'},
\end{equation*}
where $dY:=(dY_1,\cdots,dY_n)^t$ and $d$ denotes the exterior derivative. In fact, $\nabla_{(r,A,\mu, \mathcal{U})}$ is compatible with the transition functions and so defines a global connection. 
\begin{rem} \label{Tinverse}
As described above, in the local expression of the connection 1-form 
\begin{equation*}
\omega _{(r,A,\mu ,\mathcal{U})}=-\frac{\mathbf{i}}{2\pi } \left( \frac{1}{r}X^t A^t +\frac{1}{r}p^t \right) dY\cdot I_{r'} -\frac{\mathbf{i}}{2\pi }\frac{1}{r}q^t (-T+\delta )^{-1} dY\cdot I_{r'},
\end{equation*}
the second term of the right hand side includes the non-singular matrix $T'=(-T+\delta )^{-1}$. Actually, this part is associated to the unitary holonomy of the unitary local system which is defined as the corresponding mirror dual object. If we try to discuss the homological mirror symmetry on the mirror pair $(\check{T}^{2n}_{J=T}:=T^{2n}_{\tilde{\omega } =T},T^{2n}_{J=T})$ which is obtained by using Definition \ref{mirror2}, for each object in the complex geometry side, the corresponding ``unitary holonomy part'' includes $T^{-1}$.
\end{rem}
The curvature form $\Omega _{(r,A,\mu, \mathcal{U})}$ of the connection $\nabla_{(r,A,\mu, \mathcal{U})}$ is expressed locally as
\begin{equation}
\Omega _{(r,A,\mu, \mathcal{U})}=-\frac{\mathbf{i}}{2\pi r}dX^t A^t dY\cdot I_{r'}, \label{curvature}
\end{equation}
where $dX:=(dX_1,\cdots,dX_n)^t$. Here, we consider the condition such that $E_{(r,A,\mu, \mathcal{U})}$ is holomorphic. We see that the following proposition holds.
\begin{proposition}\label{holomorphic1}
For a given quadruple $(r,A,p,q)\in \mathbb{N}\times M(n;\mathbb{Z})\times \mathbb{R}^n\times \mathbb{R}^n$, the complex vector bundle $E_{(r,A,\mu, \mathcal{U})}\rightarrow T^{2n}_{J=T'}$ is holomorphic if and only if $AT'=(AT')^t$ holds.
\end{proposition}
\begin{proof}
A complex vector bundle is holomorphic if and only if the (0,2)-part of its curvature form vanishes, so we calculate the (0,2)-part of $\Omega _{(r,A,\mu ,\mathcal{U})}$. It turns out to be
\begin{equation*}
\Omega ^{(0,2)}_{(r,A,\mu, \mathcal{U})}=\frac{\mathbf{i}}{2\pi r}d\bar{Z}^t\{ T'(T'-\bar{T'} )^{-1} \}^t A^t (T'-\bar{T'})^{-1}d\bar{Z}\cdot I_{r'} ,
\end{equation*}
where $d\bar{Z} :=(d\bar{Z}_1,\cdots,d\bar{Z}_n)^t$. Thus, $\Omega ^{(0,2)}_{(r,A,\mu, \mathcal{U})}=0$ is equivalent to that $\{ T'(T'-\bar{T'} )^{-1} \}^t A^t (T'-\bar{T'} )^{-1}$ is a symmetric matrix, i.e., $AT'=(AT')^t$.
\end{proof}
Although we omit the expression here, these holomorphic vector bundles $(E_{(r,A,\mu ,\mathcal{U})},\nabla_{(r,A,\mu ,\mathcal{U})})$ form a DG-category 
\begin{equation*}
DG_{T^{2n}_{J=T'}}.
\end{equation*}
The details of this DG-category $DG_{T^{2n}_{J=T'}}$ are described in subsection 3.1 of \cite{kazushi4} (see also section 4 of \cite{kajiura}). In general, for any $A_{\infty}$-category $\mathscr{C}$, we can construct a triangulated category $Tr(\mathscr{C})$ by using the Bondal-Kapranov-Kontsevich construction \cite{bondal}, \cite{Kon}. We expect that this $DG_{T^{2n}_{J=T'}}$ generates the bounded derived category of coherent sheaves $D^b(Coh(T^{2n}_{J=T'}))$ on $T^{2n}_{J=T'}$ in the sense of the Bondal-Kapranov-Kontsevich construction, i.e.,
\begin{equation*}
Tr(DG_{T^{2n}_{J=T'}})\cong D^b(Coh(T^{2n}_{J=T'})).
\end{equation*}
At least, it is known that it split generates $D^b(Coh(T^{2n}_{J=T'}))$ when $T^{2n}_{J=T'}$ is an abelian variety (cf. \cite{orlov}, \cite{abouzaid}).

Next, we explain the symplectic geometry side, namely, define mirror dual objects corresponding to holomorphic vector bundles $E_{(r,A,\mu, \mathcal{U})} \rightarrow T^{2n}_{J=T'}$. We need to consider the Fukaya category in the symplectic geometry side, so below, we recall the definition of objects of the Fukaya categories following \cite[Definition 1.1]{Fuk}. Let $(M,\Omega)$ be a symplectic manifold $(M,\omega)$ together with a closed 2-form $B$ on $M$. Here, we put $\Omega=\omega+\sqrt{-1}B$ (note $-B+\sqrt{-1}\omega$ is used in many of the literatures). Then, we consider pairs $(L,\mathcal{L})$ with the following properties :
\begin{align}
&L \ \rm{is \ a \ Lagrangian \ submanifold \ of} \ (\it{M},\omega). \label{f1} \\ 
&\mathcal{L}\rightarrow L \ \rm{is \ a \ line \ bundle \ together \ with \ a \ connection} \ \nabla^{\mathcal{L}} \ \rm{such \ that} \label{f2} \\
&F_{\nabla^{\mathcal{L}}}=2\pi \sqrt{-1}B|_L. \notag
\end{align}
In this context, $F_{\nabla^{\mathcal{L}}}$ denotes the curvature form of the connection $\nabla^{\mathcal{L}}$. We define objects of the Fukaya category on $(M,\Omega)$ by pairs $(L,\mathcal{L})$ which satisfy the properties (\ref{f1}), (\ref{f2}). Let us consider the following $n$-dimensional submanifold $L_{(r,A,p)}$ in $\check{T}^{2n}_{J=T'}$ :
\begin{equation*}
L_{(r,A,p)}:= \biggl\{ \left( \begin{array}{ccc} \check{X} \\ \check{Y} \end{array} \right)\in \check{T}^{2n}_{J=T'} \ | \ \check{Y}=\frac{1}{r}A\check{X}+\frac{1}{r}p \biggr\}.
\end{equation*}
This $n$-dimensional submanifold $L_{(r,A,p)}$ satisfies the property (\ref{f1}), namely, $L_{(r,A,p)}$ becomes a Lagrangian submanifold in $\check{T}^{2n}_{J=T'}$ if and only if $\mathrm{Im}(-(T'^{-1})^t)A=(\mathrm{Im}(-(T'^{-1})^t)A)^t$ holds. We then consider the trivial complex line bundle $\mathcal{L}_{(r,A,p,q)}\rightarrow L_{(r,A,p)}$ with the flat connection
\begin{equation*}
\nabla_{\mathcal{L}_{(r,A,p,q)}}:=d-\frac{\mathbf{i}}{2\pi}\frac{1}{r}q^t d\check{X}.
\end{equation*}
Note that $q\in \mathbb{R}^n$ is the unitary holonomy of $\mathcal{L}_{(r,A,p,q)}$ along $L_{(r,A,p)}\approx T^n$. We discuss the property (\ref{f2}) for this pair $(L_{(r,A,p)},\mathcal{L}_{(r,A,p,q)})$ : 
\begin{equation*}
\Omega _{\mathcal{L}_{(r,A,p,q)}}=d\check{X}^t \mathrm{Re}(-(T'^{-1})^t)d\check{Y}\mid _{L_{(r,A,p)}}.
\end{equation*}
Here, $\Omega _{\mathcal{L}_{(r,A,\mu ,\mathcal{U})}}$ is the curvature form of the flat connection $\nabla_{\mathcal{L}_{(r,A,p,q)}}$ on $\mathcal{L}_{(r,A,p,q)}$, i.e., $\Omega _{\mathcal{L}_{(r,A,p,q)}}=0$. Hence, we see 
\begin{equation*}
d\check{X}^t \mathrm{Re}(-(T'^{-1})^t) d\check{Y}\mid _{L_{(r,A,p)}} = \frac{1}{r}d\check{X}^t \mathrm{Re}(-(T'^{-1})^t) Ad\check{X} =0,
\end{equation*}
so one has $\mathrm{Re}(-(T'^{-1})^t) A =(\mathrm{Re}(-(T'^{-1})^t) A)^t$. Note that $\mathrm{Im}(-(T'^{-1})^t)A=(\mathrm{Im}(-(T'^{-1})^t)A)^t$ and $\mathrm{Re}(-(T'^{-1})^t) A =(\mathrm{Re}(-(T'^{-1})^t) A)^t$ hold if and only if $AT'=(AT')^t$ holds. Thus, we obtain the following proposition.
\begin{proposition} \label{Fukobject1}
For a given quadruple $(r,A,p,q)\in \mathbb{N} \times M(n;\mathbb{Z}) \times \mathbb{R}^n \times \mathbb{R}^n$, $(L_{(r,A,p)}, \mathcal{L}_{(r,A,p,q)})$ gives an object of the Fukaya category on $\check{T}^{2n}_{J=T'}$ if and only if $AT'=(AT')^t$ holds.
\end{proposition}
Hereafter, we denote the full subcategory of the Fukaya category on $\check{T}^{2n}_{J=T'}$ consisting of objects $(L_{(r,A,p)},\mathcal{L}_{(r,A,p,q)})$ which satisfy the condition $AT'=(AT')^t$ by 
\begin{equation*}
Fuk_{\rm aff}(\check{T}^{2n}_{J=T'}).
\end{equation*}
We also give rough explanations from the viewpoint of the SYZ construction. We can regard the complexified symplectic torus $\check{T}^{2n}_{J=T'}$ as the trivial special Lagrangian torus fibration $\check{\pi} : \check{T}^{2n}_{J=T'}\rightarrow \mathbb{R}^n/2\pi\mathbb{Z}^n$, where $\check{X}$ is the local coordinates of the base space $\mathbb{R}^n/2\pi\mathbb{Z}^n$ and $\check{Y}$ is the local coordinates of the fiber of $\check{\pi} : \check{T}^{2n}_{J=T'}\rightarrow \mathbb{R}^n/2\pi\mathbb{Z}^n$. Then, we can interpret each affine Lagrangian submanifold $L_{(r,A,p)}$ in $\check{T}^{2n}_{J=T'}$ as the affine Lagrangian multi section
\begin{equation*}
s(\check{X})=\frac{1}{r}A\check{X}+\frac{1}{r}p
\end{equation*}
of $\check{\pi} : \check{T}^{2n}_{J=T'}\rightarrow \mathbb{R}^n/2\pi\mathbb{Z}^n$. Furthermore, while $r':=r_1'\cdots r_s'\in \mathbb{N}$ is the rank of $E_{(r,A,\mu,\mathcal{U})}\rightarrow T^{2n}_{J=T'}$ (see the relations (\ref{matAB}) and (\ref{r'})), in the symplectic geometry side, this $r'\in \mathbb{N}$ is interpreted as the multiplicity of $s(\check{X})=\frac{1}{r}A\check{X}+\frac{1}{r}p$.

Concerning the above discussions (in particular, Proposition \ref{holomorphic1} and Proposition \ref{Fukobject1}), we can expect that two $A_{\infty}$-categories $DG_{T^{2n}_{J=T'}}$ and $Fuk_{\rm aff}(\check{T}^{2n}_{J=T'})$ play an important role in the proof of the homological mirror symmetry conjecture on $(T^{2n}_{J=T'}, \check{T}^{2n}_{J=T'})$. Actually, it is known that the following theorem holds (\cite[Theorem 5.1]{kazushi4}). Note that two notations $p(\theta)$ and $q(\xi)$ in the following theorem are defined as follows. For elements $V_j$, $U_k\in \mathcal{U}$ $(j,k=1,\cdots, n)$, we define $\xi_j$, $\theta_k\in \mathbb{R}$ by
\begin{equation*}
e^{\mathbf{i}\xi_j}=\mathrm{det}V_j, \ e^{\mathbf{i}\theta_k}=\mathrm{det}U_k,
\end{equation*}
and set
\begin{equation*}
\xi:=(\xi_1,\cdots, \xi_n)^t, \ \theta:=(\theta_1,\cdots, \theta_n)^t.
\end{equation*}
Then, we define
\begin{equation*}
p(\theta):=p-\frac{r}{r'}\theta, \ q(\xi):=q+\frac{r}{r'}\xi.
\end{equation*}
\begin{theo}\label{bijection1}
A map $\mathrm{Ob}(DG_{T^{2n}_{J=T'}}) \rightarrow \mathrm{Ob}(Fuk_{\rm aff}(\check{T}^{2n}_{J=T'}))$ is defined by
\begin{equation*}
E_{(r,A,\mu, \mathcal{U})} \mapsto (L_{(r,A,p(\theta))}, \mathcal{L}_{(r,A,p(\theta),q(\xi))}),
\end{equation*}
and it induces a bijection between $\mathrm{Ob}^{isom}(DG_{T^{2n}_{J=T'}})$ and $\mathrm{Ob}^{isom}(Fuk_{\rm aff}(\check{T}^{2n}_{J=T'}))$, where $\mathrm{Ob}^{isom}(DG_{T^{2n}_{J=T'}})$ and $\mathrm{Ob}^{isom}(Fuk_{\rm aff}(\check{T}^{2n}_{J=T'}))$ denote the set of the isomorphism classes of objects of $DG_{T^{2n}_{J=T'}}$ and the set of the isomorphism classes of objects of $Fuk_{\rm aff}(\check{T}^{2n}_{J=T'})$, respectively\footnote{We consider affine Lagrangian submanifolds only in this paper, so two objects $(L_{(r,A,p)}, \mathcal{L}_{(r,A,p,q)})$, $(L_{(s,B,u)}, \mathcal{L}_{(s,B,u,v)})\in Fuk_{\rm aff}(\check{T}^{2n}_{J=T'})$ are isomorphic to each other if and only if $L_{(r,A,p)}=L_{(s,B,u)}$ and $\mathcal{L}_{(r,A,p,q)}\cong \mathcal{L}_{(s,B,u,v)}$.}.
\end{theo}

Now, we explain the homological mirror symmetry setting for $(\check{T}^{2n}_{J=T},T^{2n}_{J=T})$. By using the biholomorphic map $\varphi : T^{2n}_{J=T} \stackrel{\sim }{\rightarrow} T^{2n}_{J=T'}$ and complex vector bundles $E_{(r,A,\mu ,\mathcal{U})}\rightarrow T^{2n}_{J=T'}$, we can consider pullback bundles 
\begin{equation*}
\varphi ^*E_{(r,A,\mu ,\mathcal{U})}\rightarrow T^{2n}_{J=T}
\end{equation*}
of rank $r'$. Then, the connection $\varphi ^*\nabla_{(r,A,\mu ,\mathcal{U})}$ on $\varphi ^*E_{(r,A,\mu ,\mathcal{U})}$ is expressed locally as
\begin{equation*}
\varphi ^*\nabla_{(r,A,\mu ,\mathcal{U})}=d-\frac{\mathbf{i}}{2\pi } \left( -\frac{1}{r}y^t A^t +\frac{1}{r}\mu ^t \right) (dx+\delta dy)\cdot I_{r'}.
\end{equation*}
In particular, by Proposition \ref{holomorphic1} and biholomorphicity of the map $\varphi : T^{2n}_{J=T} \stackrel{\sim }{\rightarrow } T^{2n}_{J=T'}$, we see that the following proposition holds.
\begin{proposition} \label{holomorphic2}
For a given quadruple $(r,A,p,q)\in \mathbb{N}\times M(n;\mathbb{Z})\times \mathbb{R}^n \times \mathbb{R}^n$, the complex vector bundle $\varphi ^*E_{(r,A,\mu ,\mathcal{U})}\rightarrow T^{2n}_{J=T}$ is holomorphic if and only if $AT'=(AT')^t$ holds.
\end{proposition}
These holomorphic vector bundles $(\varphi ^*E_{(r,A,\mu ,\mathcal{U})},\varphi ^*\nabla_{(r,A,\mu ,\mathcal{U})})$ again form a DG-category 
\begin{equation*}
\mathcal{DG}^{\varphi }_{T^{2n}_{J=T}}. 
\end{equation*}
In particular, $DG_{T^{2n}_{J=T'}}$ is equivalent to $\mathcal{DG}^{\varphi }_{T^{2n}_{J=T}}$ as DG-categories. 

Here, we explain the nonuniqueness of the choice of a matrix $\delta $ in Lemma \ref{delta} and related discussions. We assume that $\delta _1$, $\delta _2\in M(n;\mathbb{Z})$ satisfy $\mathrm{det}(T-\delta _1)\not=0$, $\mathrm{det}(T-\delta _2)\not=0$, respectively, and set 
\begin{equation*}
T_1':=(-T+\delta _1)^{-1}, \ T_2':=(-T+\delta _2)^{-1}. 
\end{equation*}
Then, biholomorphic maps 
\begin{equation*}
\varphi _1 : T^{2n}_{J=T} \stackrel{\sim }{\rightarrow } T^{2n}_{J=T_1'}, \ \varphi _2 : T^{2n}_{J=T} \stackrel{\sim }{\rightarrow } T^{2n}_{J=T_2'}
\end{equation*}
are defined by 
\begin{equation*}
\varphi _1(z)=(-T+\delta _1)^{-1}z, \ \varphi _2(z)=(-T+\delta _2)^{-1}z,
\end{equation*} 
respectively. Furthermore, we can also define four DG-categories $DG_{T^{2n}_{J=T_1'}}$, $DG_{T^{2n}_{J=T_2'}}$, $\mathcal{DG}^{\varphi _1}_{T^{2n}_{J=T}}$ and $\mathcal{DG}^{\varphi _2}_{T^{2n}_{J=T}}$. In this situation, we can prove the following proposition.
\begin{proposition} \label{derived}
For $i=1$, $2$, if $D^b(Coh(T^{2n}_{J=T_i'}))$ is generated by $DG_{T^{2n}_{J=T_i'}}$, there exists an equivalence $Tr(\mathcal{DG}^{\varphi _1}_{T^{2n}_{J=T}})\cong Tr(\mathcal{DG}^{\varphi _2}_{T^{2n}_{J=T}})$ as triangulated categories.
\end{proposition}
\begin{proof}
First, for given $A_{\infty}$-categories $\mathscr{C}_1$, $\mathscr{C}_2$ which are equivalent to each other, note that there exists an equivalence $Tr(\mathscr{C}_1)\cong Tr(\mathscr{C}_2)$ as triangulated categories. Since $\mathcal{DG}^{\varphi _1}_{T^{2n}_{J=T}}\cong DG_{T^{2n}_{J=T_1'}}$, we see
\begin{equation*}
Tr(\mathcal{DG}^{\varphi _1}_{T^{2n}_{J=T}}) \cong Tr(DG_{T^{2n}_{J=T_1'}}).
\end{equation*}
Furthermore, by assumption, we have
\begin{equation*}
Tr(DG_{T^{2n}_{J=T_1'}})\cong D^b(Coh(T^{2n}_{J=T_1'})).
\end{equation*}
On the other hand, since $T^{2n}_{J=T_1'}$ is biholomorphic to $T^{2n}_{J=T}$, one obtains
\begin{equation*}
D^b(Coh(T^{2n}_{J=T_1'}))\cong D^b(Coh(T^{2n}_{J=T})).
\end{equation*}
Hence, 
\begin{equation}
Tr(\mathcal{DG}^{\varphi _1}_{T^{2n}_{J=T}})\cong D^b(Coh(T^{2n}_{J=T})) \label{derived1}
\end{equation}
holds. Similarly, we can prove
\begin{equation}
Tr(\mathcal{DG}^{\varphi _2}_{T^{2n}_{J=T}})\cong D^b(Coh(T^{2n}_{J=T})), \label{derived2}
\end{equation}
and by using two relations (\ref{derived1}), (\ref{derived2}), we see that
\begin{equation*}
Tr(\mathcal{DG}^{\varphi _1}_{T^{2n}_{J=T}})\cong Tr(\mathcal{DG}^{\varphi _2}_{T^{2n}_{J=T}})
\end{equation*}
holds.
\end{proof}

We define mirror dual objects corresponding to holomorphic vector bundles $\varphi ^*E_{(r,A,\mu ,\mathcal{U})}\rightarrow T^{2n}_{J=T}$. Let us denote the local coordinates of $\check{T}^{2n}_{J=T}$ by $(x^1,\cdots, x^n, y^1,\cdots, y^n)^t$, and we define
\begin{equation*}
\check{x}:=(x^1,\cdots, x^n)^t,\ \check{y}:=(y^1,\cdots, y^n)^t.
\end{equation*}
We consider the following $n$-dimensional submanifold $\tilde{L} _{(r,A,p)}$ in $\check{T}^{2n}_{J=T}$ :
\begin{equation*}
\tilde{L} _{(r,A,p)}:= \biggl\{ \left( \begin{array}{ccc} \check{x} \\ \check{y} \end{array} \right) \in \check{T}^{2n}_{J=T} \ | \ \check{x}=-\frac{1}{r}A\check{y}+\frac{1}{r}p \biggr\}.
\end{equation*}
By a direct calculation, we see that this $n$-dimensional submanifold $\tilde{L} _{(r,A,p)}$ satisfies the property (\ref{f1}), namely, $\tilde{L}_{(r,A,p)}$ becomes a Lagrangian submanifold in $\check{T}^{2n}_{J=T}$ if and only if $(\mathrm{Im}(T-\delta ))^t (-A)=((\mathrm{Im}(T-\delta ))^t (-A))^t$, i.e., $(\mathrm{Im}T)^t A=((\mathrm{Im}T)^t A)^t$ holds. We then consider the trivial complex line bundle $\tilde{\mathcal{L}} _{(r,A,p,q)}\rightarrow \tilde{L} _{(r,A,p)}$ with the flat connection
\begin{equation*}
\nabla_{\tilde{\mathcal{L}}_{(r,A,p,q)}}:=d+\frac{\mathbf{i}}{2\pi}\frac{1}{r}q^t d\check{y}.
\end{equation*}
Note that $q\in \mathbb{R}^n$ is the unitary holonomy of $\tilde{\mathcal{L}}_{(r,A,p,q)}$ along $\tilde{L} _{(r,A,p)}\approx T^n$. Also in this case, we discuss the property (\ref{f2}) for this pair $(\tilde{L}_{(r,A,p)}, \tilde{\mathcal{L}}_{(r,A,p,q)})$ : 
\begin{equation*}
\Omega _{\tilde{\mathcal{L}} _{(r,A,p,q)}}=d\check{x}^t \mathrm{Re}(T-\delta ) d\check{y} \mid _{\tilde{L} _{(r,A,p)}}.
\end{equation*}
Here, $\Omega _{\tilde{\mathcal{L}} _{(r,A,p,q)}}$ is the curvature form of the flat connection $\nabla_{\tilde{\mathcal{L}}_{(r,A,p,q)}}$ on $\tilde{\mathcal{L}}_{(r,A,p,q)}$, i.e., $\Omega _{\tilde{\mathcal{L}}_{(r,A,p,q)}}=0$. Hence, we see 
\begin{equation*}
d\check{x}^t \mathrm{Re}(T-\delta ) d\check{y}\mid _{\tilde{L} _{(r,A,p)}}=-\frac{1}{r}d\check{y}^t A^t \mathrm{Re}(T-\delta ) d\check{y} =0,
\end{equation*}
so one has $A^t \mathrm{Re}(T-\delta )=(A^t \mathrm{Re}(T-\delta ))^t$. Note that $(\mathrm{Im}T)^t A=((\mathrm{Im}T)^t A)^t$ and $A^t \mathrm{Re}(T-\delta )=(A^t \mathrm{Re}(T-\delta ))^t$ hold if and only if $AT'=(AT')^t$ holds. Thus, we obtain the following proposition.
\begin{proposition} \label{Fukobject2}
For a given quadruple $(r,A,p,q)\in \mathbb{N} \times M(n;\mathbb{Z}) \times \mathbb{R}^n \times \mathbb{R}^n$, $(\tilde{L}_{(r,A,p)},\tilde{\mathcal{L}}_{(r,A,p,q)})$ gives an object of the Fukaya category on $\check{T}^{2n}_{J=T}$ if and only if $AT'=(AT')^t$ holds.
\end{proposition}
Hereafter, we denote the full subcategory of the Fukaya category on $\check{T}^{2n}_{J=T}$ consisting of objects $(\tilde{L}_{(r,A,p)},\tilde{\mathcal{L}}_{(r,A,p,q)})$ which satisfy the condition $AT'=(AT')^t$ by 
\begin{equation*}
Fuk_{\rm aff}(\check{T}^{2n}_{J=T}). 
\end{equation*}

As expected from the statements of Proposition \ref{holomorphic2} and Proposition \ref{Fukobject2}, for two $A_{\infty}$-categories $\mathcal{DG}_{T^{2n}_{J=T}}^{\varphi}$ and $Fuk_{\rm aff}(\check{T}^{2n}_{J=T})$, we can obtain the following analogue result of Theorem \ref{bijection1}. Here, we omit the proof of the following theorem because it is proved in a similar way as in the proof of Theorem \ref{bijection1}.
\begin{theo}\label{bijection2}
A map $\mathrm{Ob}(\mathcal{DG}_{T^{2n}_{J=T}}^{\varphi}) \rightarrow \mathrm{Ob}(Fuk_{\rm aff}(\check{T}^{2n}_{J=T}))$ is defined by
\begin{equation*}
\varphi^*E_{(r,A,\mu, \mathcal{U})} \mapsto (\tilde{L}_{(r,A,p(\theta))}, \tilde{\mathcal{L}}_{(r,A,p(\theta),q(\xi))}),
\end{equation*}
and it induces a bijection between $\mathrm{Ob}^{isom}(\mathcal{DG}_{T^{2n}_{J=T}}^{\varphi})$ and $\mathrm{Ob}^{isom}(Fuk_{\rm aff}(\check{T}^{2n}_{J=T}))$, where $\mathrm{Ob}^{isom}(\mathcal{DG}_{T^{2n}_{J=T}}^{\varphi})$ and $\mathrm{Ob}^{isom}(Fuk_{\rm aff}(\check{T}^{2n}_{J=T}))$ denote the set of the isomorphism classes of objects of $\mathcal{DG}_{T^{2n}_{J=T}}^{\varphi}$ and the set of the isomorphism classes of objects of $Fuk_{\rm aff}(\check{T}^{2n}_{J=T})$, respectively.
\end{theo}

By summarizing the above discussions, we obtain the following diagrams. The left hand side of the following diagrams is the our construction which is proposed in sections 3, 4, and the right hand side of the following diagrams is based on the SYZ construction.
\begin{equation*}
\begin{CD}
\varphi ^*E_{(r,A,\mu ,\mathcal{U})}                                                                      @>>>            E_{(r,A,\mu ,\mathcal{U})} \\
@VVV                                                                                                                                  @VVV \\
T^{2n}_{J=T}=\mathbb{C}^n/2\pi (\mathbb{Z}^n\oplus T\mathbb{Z}^n)            @>\sim >\varphi >      T^{2n}_{J=T'}=\mathbb{C}^n/2\pi (\mathbb{Z}^n\oplus T'\mathbb{Z}^n) \\
@V\mathrm{mirror\ dual}VV                                                                                                    @VV\mathrm{Definition} \ \ref{mirror1}V \\
\check{T}^{2n}_{J=T}=(T^{2n},\tilde{\omega }=d\check{x}^t (T-\delta ) d\check{y})          @.             \check{T}^{2n}_{J=T'}=(T^{2n},\tilde{\omega }=d\check{X}^t(-(T'^{-1})^t)d\check{Y}), \vspace{5mm} \\
\varphi ^*E_{(r,A,\mu ,\mathcal{U})}                                                            @>AT'=(AT')^t>>       E_{(r,A,\mu ,\mathcal{U})} \\
@V\mathrm{mirror\ dual}VAT'=(AT')^tV                                                                                    @VAT'=(AT')^tV\mathrm{mirror\ dual}V \\
(\tilde{L}_{(r,A,p(\theta))},\tilde{\mathcal{L}}_{(r,A,p(\theta),q(\xi))})                                        @.             (L_{(r,A,p(\theta))},\mathcal{L}_{(r,A,p(\theta),q(\xi))}).
\end{CD}
\end{equation*}

\section{On the class of holomorphic vector bundles $\varphi ^*E_{(r,A,\mu ,\mathcal{U})}^{T'}$}
Throughout this section, we denote $E_{(r,A,\mu,\mathcal{U})}\rightarrow T^{2n}_{J=T'}$ by
\begin{equation*}
E_{(r,A,\mu,\mathcal{U})}^{T'}\rightarrow T^{2n}_{J=T'}
\end{equation*}
in order to specify that the definition of $E_{(r,A,\mu,\mathcal{U})}$ depends on the complex structure $T'$ of $T^{2n}_{J=T'}$. In this section, we investigate holomorphic vector bundles $\varphi^*E_{(r,A,\mu,\mathcal{U})}^{T'}\rightarrow T^{2n}_{J=T}$. These holomorphic vector bundles $\varphi ^*E_{(r,A,\mu ,\mathcal{U})}^{T'}$ are examples of projectively flat bundles, and in general, the factors of automorphy of projectively flat bundles on complex tori are classified concretely \cite{Hano}, \cite{matsu}, \cite{koba}, \cite{yang}. Hence, by using this classification result, we first interpret holomorphic vector bundles $\varphi ^*E_{(r,A,\mu ,\mathcal{U})}^{T'}$ in the language of factors of automorphy (Theorem \ref{maintheorem}). On the other hand, as an analogue of holomorphic vector bundles $E_{(r,A,\mu,\mathcal{U})}^{T'}$, we can also consider holomorphic vector bundles 
\begin{equation*}
E_{(r,A,\mu,\mathcal{U})}^T\rightarrow T^{2n}_{J=T}
\end{equation*}
by regarding $T'$ as $T$. Thus, finally, we check that the set consisting of projectively flat bundles $\varphi ^*E_{(r,A,\mu ,\mathcal{U})}^{T'}$ differs from the set consisting of holomorphic vector bundles $E_{(r,A,\mu,\mathcal{U})}^T$, and consider the meaning of this result.

Let us define $L:=2\pi (\mathbb{Z}^n\oplus T\mathbb{Z}^n)$, i.e., $T^{2n}_{J=T}=\mathbb{C}^n/L$. Then, this lattice $L$ in $\mathbb{C}^n$ is generated by
\begin{gather*}
\gamma _1:=(2\pi ,0,\cdots, 0)^t, \cdots, \gamma _n:=(0,\cdots, 0, 2\pi )^t, \\
\gamma '_1:=(2\pi t_{11},\cdots, 2\pi t_{n1})^t, \cdots, \gamma '_n:=(2\pi t_{1n},\cdots, 2\pi t_{nn})^t,
\end{gather*}
where $t_{ij}$ denotes the $(i,j)$ component of $T$.

Now, we discuss the relations between holomorphic vector bundles $\varphi ^*E_{(r,A,\mu ,\mathcal{U})}^{T'}$ and projectively flat bundles. By a direct calculation, we see that the curvature form $\varphi ^*\Omega _{(r,A,\mu ,\mathcal{U})}$ of a holomorphic vector bundle $\varphi ^*E_{(r,A,\mu ,\mathcal{U})}^{T'}$ is expressed locally as 
\begin{equation}
\varphi ^*\Omega _{(r,A,\mu ,\mathcal{U})}=\frac{\mathbf{i}}{2\pi r'}\frac{r'}{r}dz^t \{ (T-\bar{T})^{-1} \}^t A^t d\bar{z}\cdot I_{r'}. \label{curvature}
\end{equation}
In general, for a compact K\"{a}hler manifold $M$ and a holomorphic vector bundle $E\rightarrow M$ of rank $r$ with a curvature form $\Omega _E$, $E$ is projectively flat if and only if there exists a complex 2-form $\alpha $ on $M$ such that $\Omega _E=\alpha \cdot I_r$ (see \cite{matsu}, \cite{koba}, \cite{yang}). Therefore, by the local expression (\ref{curvature}), it is clear that $\varphi ^*E_{(r,A,\mu ,\mathcal{U})}^{T'}$ is projectively flat. Here, we set
\begin{equation*}
R:=\frac{\mathbf{i}}{2\pi }\frac{r'}{r}\{ (T-\bar{T})^{-1} \}^t A^t,
\end{equation*}
namely,
\begin{equation*}
\varphi ^*\Omega _{(r,A,\mu ,\mathcal{U})}=\frac{1}{r'}dz^t R d\bar{z}\cdot I_{r'}.
\end{equation*}
Then, the following lemma holds.
\begin{lemma}
The matrix $R$ is a real symmetric matrix of order $n$.
\end{lemma}
\begin{proof}
By a direct calculation, 
\begin{align*}
R&=\frac{\mathbf{i}}{2\pi }\frac{r'}{r}\{ (T-\bar{T})^{-1} \}^t A^t (T-\bar{T})(T-\bar{T})^{-1} \\
&=\frac{\mathbf{i}}{2\pi }\frac{r'}{r}\{ (T-\bar{T})^{-1} \}^t A^t (T-\delta )(T-\bar{T})^{-1}-\frac{\mathbf{i}}{2\pi }\frac{r'}{r}\{ (T-\bar{T})^{-1} \}^t A^t (\bar{T}-\delta )(T-\bar{T})^{-1},
\end{align*}
and since $AT'=(AT')^t$, i.e., $(T-\delta )^t A=A^t (T-\delta )$, it is clear that the two matrices
\begin{equation*}
\frac{\mathbf{i}}{2\pi }\frac{r'}{r}\{ (T-\bar{T})^{-1} \}^t A^t (T-\delta )(T-\bar{T})^{-1},\ \frac{\mathbf{i}}{2\pi }\frac{r'}{r}\{ (T-\bar{T})^{-1} \}^t A^t (\bar{T}-\delta )(T-\bar{T})^{-1}
\end{equation*}
are symmetric. Hence, $R$ is a symmetric matrix. Furthermore, when we denote $T=T_R+\mathbf{i}T_I$ with $T_R:=\mathrm{Re}T$, $T_I:=\mathrm{Im}T$, one has
\begin{equation*}
R=\frac{1}{4\pi }\frac{r'}{r}(T_I^{-1})^t A^t.
\end{equation*}
This relation indicates $R\in M(n;\mathbb{R})$.
\end{proof}
\begin{rem}
Although the matrix $R$ is defined by using the matrix $\frac{r'}{r}A$, each component of the matrix $\frac{r'}{r}A$ is an integer.
\end{rem}
By using this real symmetric matrix $R=(R_{ij})$ of order $n$, we define an Hermitian bilinear form $\mathcal{R} : \mathbb{C}^n\times \mathbb{C}^n\rightarrow \mathbb{C}$ by
\begin{equation*}
\mathcal{R}(z,w):=\sum_{i,j=1}^{n} R_{ij} z_i \bar{w_j},
\end{equation*}
where $z=(z_1,\cdots,z_n)^t$, $w=(w_1,\cdots,w_n)^t$. Then, the following propositions hold.
\begin{proposition}
For $\gamma _1,\cdots, \gamma _n$, $\mathrm{Im}\mathcal{R}(\gamma _j,\gamma _k)=0$, where $j,k=1,\cdots,n$.
\end{proposition}
\begin{proof}
By definition, $\mathcal{R}(\gamma _j,\gamma _k)=4\pi ^2 R_{jk}$, where $R_{jk}\in \mathbb{R}$, so $\mathrm{Im}\mathcal{R}(\gamma _j,\gamma _k)=0$.
\end{proof}
\begin{proposition}
For $\gamma '_1,\cdots, \gamma '_n$, $\mathrm{Im}\mathcal{R}(\gamma '_j,\gamma '_k)=(\pi \frac{r'}{r}(A^t \delta -\delta ^t A))_{jk}$, where $j,k=1,\cdots,n$.
\end{proposition}
\begin{proof}
By definition, $\mathcal{R}(\gamma '_j, \gamma '_k)=(4\pi ^2T^t R \bar{T})_{jk}$, so for $T=T_R+\mathbf{i}T_I$, it turns out to be 
\begin{align*}
4\pi ^2 T^t R \bar{T}&=4\pi ^2 (T_R^t +\mathbf{i}T_I^t)\cdot \frac{1}{4\pi }\frac{r'}{r}(T_I^{-1})^t A^t \cdot (T_R-\mathbf{i}T_I) \\
&=\pi \frac{r'}{r}\{ T_R^t (T_I^{-1})^t A^t T_R +A^t T_I +\mathbf{i}(A^t T_R -T_R^t (T_I^{-1})^t A^t T_I) \}.
\end{align*}
Thus, 
\begin{align*}
\mathrm{Im}\mathcal{R}(\gamma '_j,\gamma '_k)&=\left( \pi \frac{r'}{r}(A^t T_R -T_R^t (T_I^{-1})^t A^t T_I) \right)_{jk} \\
&=\left( \pi \frac{r'}{r}(A^t T_R -T_R^t A) \right)_{jk} \\
&=\left( \pi \frac{r'}{r}(A^t \delta -\delta ^t A) \right)_{jk}.
\end{align*}
Here, the second equality and the third equality follow from $AT'=(AT')^t$.
\end{proof}
\begin{proposition}
For $\gamma _1,\cdots, \gamma _n$ and $\gamma '_1,\cdots, \gamma '_n$, $\mathrm{Im}\mathcal{R}(\gamma _j,\gamma '_k)=-\pi \frac{r'}{r}a_{jk}$, $\mathrm{Im}\mathcal{R}(\gamma '_k,\gamma _j)=\pi \frac{r'}{r}a_{jk}$, where $j,k=1,\cdots,n$.
\end{proposition}
\begin{proof}
First, we prove $\mathrm{Im}R\bar{T}=-\frac{1}{4\pi }\frac{r'}{r}A$. For $T=T_R+\mathbf{i}T_I$,
\begin{equation*}
R\bar{T}=\frac{1}{4\pi }\frac{r'}{r}(T_I^{-1})^t A^t T_R -\frac{\mathbf{i}}{4\pi }\frac{r'}{r}(T_I^{-1})^t A^t T_I,
\end{equation*}
so we see 
\begin{align*}
\mathrm{Im}R\bar{T}&=-\frac{1}{4\pi }\frac{r'}{r}(T_I^{-1})^t A^t T_I \\
&=-\frac{1}{4\pi }\frac{r'}{r}A.
\end{align*}
Here, we used $AT'=(AT')^t$. Similarly, we can also prove 
\begin{equation*}
\mathrm{Im}\bar{R}T=\frac{1}{4\pi }\frac{r'}{r}A.
\end{equation*}
On the other hand, the following relations hold.
\begin{equation*}
\mathcal{R}(\gamma _j,\gamma '_k)=(4\pi ^2 R\bar{T})_{jk},\ \mathcal{R}(\gamma '_k,\gamma _j)=(4\pi ^2 \bar{R}T)_{jk}.
\end{equation*}
Thus, by using $\mathrm{Im}R\bar{T}=-\frac{1}{4\pi }\frac{r'}{r}A$ and $\mathrm{Im}\bar{R}T=\frac{1}{4\pi }\frac{r'}{r}A$, we obtain
\begin{equation*}
\mathrm{Im}\mathcal{R}(\gamma _j,\gamma '_k)=-\pi \frac{r'}{r}a_{jk},\ \mathrm{Im}\mathcal{R}(\gamma '_k,\gamma _j)=\pi \frac{r'}{r}a_{jk}.
\end{equation*}
\end{proof}
Now, we consider a projectively flat bundle $\mathcal{E}_{(r,A,\mu ,\mathcal{U})}\rightarrow T^{2n}_{J=T}$ of rank $r'$ whose factor of automorphy $j : L\times \mathbb{C}^n \rightarrow GL(r';\mathbb{C})$ and connection $\tilde{\nabla}_{(r,A,\mu ,\mathcal{U})} =d+\tilde{\omega }_{(r,A,\mu ,\mathcal{U})} $ are expressed locally as follows.
\begin{equation*}
j(\gamma ,z)=U(\gamma )\mathrm{exp} \left\{ \frac{1}{r'}\mathcal{R}(z,\gamma )+\frac{1}{2r'} \mathcal{R}(\gamma ,\gamma ) \right\}, 
\end{equation*}
\begin{align*}
\tilde{\omega }_{(r,A,\mu ,\mathcal{U})} =-\frac{1}{r'}dz^t R \bar{z}\cdot I_{r'} &-\frac{\mathbf{i}}{2\pi r}\mu ^t (\delta -\bar{T})(T-\bar{T})^{-1}dz\cdot I_{r'} \\
&+\frac{\mathbf{i}}{2\pi r} \bar{\mu }^t (\delta -\bar{T})(T-\bar{T})^{-1} dz\cdot I_{r'}.
\end{align*}
Here, $U(\gamma _j)$, $U(\gamma '_k)\in U(r')$ $(j,k=1,\cdots,n)$ satisfy the relations
\begin{align}
&U(\gamma _j)U(\gamma _k)=U(\gamma _k)U(\gamma _j), \label{cocycle1} \\
&U(\gamma '_j)U(\gamma '_k)=\zeta ^{(A^t \delta )_{jk}-(A^t \delta )_{kj}}U(\gamma '_k)U(\gamma '_j), \label{cocycle2} \\
&\zeta ^{-a_{jk}}U(\gamma '_k)U(\gamma _j)=U(\gamma _j)U(\gamma '_k), \label{cocycle3}
\end{align}
and
\begin{equation*}
\mathcal{U}:= \Bigl\{ U(\gamma _j),\ U(\gamma '_k) \in U(r') \ | \ (\ref{cocycle1}),\ (\ref{cocycle2}),\ (\ref{cocycle3}),\ j,k=1,\cdots, n \Bigr\}.
\end{equation*}
In order to compare the definition of $\varphi ^*E_{(r,A,\mu ,\mathcal{U})}^{T'}$ with the definition of $\mathcal{E}_{(r,A,\mu ,\mathcal{U})}$, we recall the definition of transition functions of $\varphi ^*E_{(r,A,\mu ,\mathcal{U})}^{T'}$. Let 
\begin{equation*}
s(x_1,\cdots,x_n,y_1,\cdots,y_n) 
\end{equation*}
be a smooth section of $\varphi ^*E_{(r,A,\mu ,\mathcal{U})}^{T'}$. Then, the transition functions of $\varphi ^*E_{(r,A,\mu ,\mathcal{U})}^{T'}$ are given by
\begin{align*}
&s(x_1,\cdots, x_j+2\pi ,\cdots, x_n, y_1,\cdots, y_n)=V'_j \cdot s(x_1,\cdots, x_n, y_1,\cdots, y_n), \\
&s(x_1,\cdots, x_n, y_1,\cdots, y_k+2\pi , \cdots, y_n)=e^{-\frac{\mathbf{i}}{r}a_k (x+\delta y)}U'_k \cdot s(x_1,\cdots, x_n, y_1,\cdots, y_n), 
\end{align*}
where $V'_j$, $U'_k \in U(r')$ $(j,k=1,\cdots, n)$ and $a_k:=(a_{1k},\cdots, a_{nk})$. In particular, the cocycle condition is expressed as 
\begin{equation*}
V'_j V'_k =V'_k V'_j,\ U'_j U'_k =\zeta ^{(A^t \delta )_{jk}-(A^t \delta )_{kj}}U'_k U'_j,\ \zeta ^{-a_{jk}}U'_k V'_j =V'_j U'_k .
\end{equation*}
Clearly, the relations (\ref{cocycle1}), (\ref{cocycle2}) and (\ref{cocycle3}) are equivalent to the cocycle condition of $\varphi ^*E_{(r,A,\mu ,\mathcal{U})}^{T'}$ (in fact, one of the purposes of this section is to prove $\varphi ^*E_{(r,A,\mu ,\mathcal{U})}\cong \mathcal{E}_{(r,A,\mu ,\mathcal{U})}^{T'}$, and this result is given in Theorem \ref{maintheorem}). Furthermore, the curvature form $\tilde{\Omega } _{(r,A,\mu ,\mathcal{U})}$ of $\mathcal{E}_{(r,A,\mu ,\mathcal{U})}$ is expressed locally as
\begin{equation*}
\tilde{\Omega } _{(r,A,\mu ,\mathcal{U})}=\frac{1}{r'}dz^t R d\bar{z}\cdot I_{r'}.
\end{equation*}
Hence, we fix $r$, $A$, $\mu $ (note that $r'$ is uniquely defined by using $r$ and $A$), and by comparing the definition of $\varphi ^*E_{(r,A,\mu ,\mathcal{U})}^{T'}$ with the definition of $\mathcal{E}_{(r,A,\mu ,\mathcal{U})}$, we see that the cardinality of the set $\{(\varphi ^*E_{(r,A,\mu ,\mathcal{U})}^{T'},\varphi ^*\nabla_{(r,A,\mu ,\mathcal{U})}) \}$ is equal to the cardinality of the set $\{ (\mathcal{E}_{(r,A,\mu ,\mathcal{U})},\tilde{\nabla} _{(r,A,\mu ,\mathcal{U})}) \}$. Thus, we expect that there exists an isomorphism $\Psi : \varphi ^*E_{(r,A,\mu ,\mathcal{U})}^{T'} \stackrel{\sim }{\rightarrow } \mathcal{E}_{(r,A,\mu ,\mathcal{U})}$ which gives a correspondence between $\{ (\varphi ^*E_{(r,A,\mu ,\mathcal{U})}^{T'}, \varphi ^*\nabla_{(r,A,\mu ,\mathcal{U})}) \}$ and $\{ (\mathcal{E}_{(r,A,\mu ,\mathcal{U})},\tilde{\nabla} _{(r,A,\mu ,\mathcal{U})}) \}$. Actually, the following theorem holds.
\begin{theo} \label{maintheorem}
One has $\varphi ^*E_{(r,A,\mu ,\mathcal{U})}\cong \mathcal{E}_{(r,A,\mu ,\mathcal{U})}^{T'}$, where an isomorphism $\Psi : \varphi ^*E_{(r,A,\mu ,\mathcal{U})}^{T'} \stackrel{\sim }{\rightarrow } \mathcal{E}_{(r,A,\mu ,\mathcal{U})}$ is expressed locally as
\begin{align*}
\Psi (z,\bar{z})=\mathrm{exp} \biggl\{ & \frac{\mathbf{i}}{4\pi r'} z^t \bar{\mathcal{A}} z +\frac{\mathbf{i}}{4\pi r'} \bar{z}^t \mathcal{A} \bar{z} -\frac{\mathbf{i}}{2\pi r'} z^t \mathcal{A} \bar{z} \\
&+\frac{\mathbf{i}}{2\pi r} \bar{z}^t \{ (T-\bar{T})^{-1} \}^t (\delta -T)^t \mu -\frac{\mathbf{i}}{2\pi r} z^t \{ (T-\bar{T})^{-1} \}^t (\delta -\bar{T})^t \bar{\mu } \biggr\} \cdot I_{r'}, 
\end{align*}
$\mathcal{A}:=\frac{r'}{r}\{ (T-\bar{T})^{-1} \}^t A^t (\delta -T) (T-\bar{T})^{-1}$.
\end{theo}
\begin{proof}
Note that $\mathcal{A}$ is a symmetric matrix because $AT'=(AT')^t$. We construct an isomorphism $\Psi : \varphi ^*E_{(r,A,\mu ,\mathcal{U})}^{T'} \stackrel{\sim }{\rightarrow } \mathcal{E}_{(r,A,\mu ,\mathcal{U})}$ explicitly such that its local expression is
\begin{equation*}
\Psi (z,\bar{z})=\psi (z,\bar{z})\cdot I_{r'},
\end{equation*}
where $\psi (z,\bar{z})$ is a function defined locally. By solving the differential equation
\begin{equation*}
\tilde{\nabla} _{(r,A,\mu ,\mathcal{U})}\Psi (z,\bar{z})=\Psi (z,\bar{z}) \varphi ^*\nabla_{(r,A,\mu ,\mathcal{U})},
\end{equation*}
we obtain the solution
\begin{align*}
\psi (z,\bar{z})=c\cdot \mathrm{exp} \biggl\{ & \frac{\mathbf{i}}{4\pi r'} z^t \bar{\mathcal{A}} z +\frac{\mathbf{i}}{4\pi r'} \bar{z}^t \mathcal{A} \bar{z} -\frac{\mathbf{i}}{2\pi r'} z^t \mathcal{A} \bar{z} \\
&+\frac{\mathbf{i}}{2\pi r} \bar{z}^t \{ (T-\bar{T})^{-1} \}^t (\delta -T)^t \mu -\frac{\mathbf{i}}{2\pi r} z^t \{ (T-\bar{T})^{-1} \}^t (\delta -\bar{T})^t \bar{\mu } \biggr\},
\end{align*}
where $c$ is an arbitrary constant, so by setting $c=1$, one has
\begin{align*}
\Psi (z,\bar{z})=\mathrm{exp} \biggl\{ & \frac{\mathbf{i}}{4\pi r'} z^t \bar{\mathcal{A}} z +\frac{\mathbf{i}}{4\pi r'} \bar{z}^t \mathcal{A} \bar{z} -\frac{\mathbf{i}}{2\pi r'} z^t \mathcal{A} \bar{z} \\
&+\frac{\mathbf{i}}{2\pi r} \bar{z}^t \{ (T-\bar{T})^{-1} \}^t (\delta -T)^t \mu -\frac{\mathbf{i}}{2\pi r} z^t \{ (T-\bar{T})^{-1} \}^t (\delta -\bar{T})^t \bar{\mu } \biggr\} \cdot I_{r'}.
\end{align*}
By using this $\Psi : \varphi ^*E_{(r,A,\mu ,\mathcal{U})}^{T'} \stackrel{\sim }{\rightarrow } \mathcal{E}_{(r,A,\mu ,\mathcal{U})}$, we transform the transition functions of $\varphi ^*E_{(r,A,\mu ,\mathcal{U})}^{T'}$. First, we prove the relation
\begin{equation}
\bar{\mathcal{A}}-\mathcal{A}=\frac{2\pi }{\mathbf{i}}R. \label{rel}
\end{equation}
We see
\begin{align}
\bar{\mathcal{A}}-\mathcal{A}&=\frac{r'}{r}\{ (T-\bar{T})^{-1} \}^t A^t (\delta -\bar{T}) (T-\bar{T})^{-1}-\{ (T-\bar{T})^{-1} \}^t A^t (\delta -T)(T-\bar{T})^{-1} \notag \\
&=\frac{r'}{r}\{ (T-\bar{T})^{-1} \}^t A^t \{ (\delta -\bar{T})-(\delta -T) \} (T-\bar{T})^{-1} \notag \\
&=\frac{r'}{r}\{ (T-\bar{T})^{-1} \}^t A^t (T-\bar{T}) (T-\bar{T})^{-1} \notag \\
&=\frac{r'}{r}\{ (T-\bar{T})^{-1} \}^t A^t. \label{rel1}
\end{align}
On the other hand, by definition,
\begin{equation*}
R=\frac{\mathbf{i}}{2\pi } \frac{r'}{r}\{ (T-\bar{T})^{-1} \}^t A^t, 
\end{equation*}
namely,
\begin{equation}
\frac{r'}{r}\{ (T-\bar{T})^{-1} \}^t A^t =\frac{2\pi }{\mathbf{i}}R. \label{rel2}
\end{equation}
Therefore, by comparing the relation (\ref{rel1}) with the relation (\ref{rel2}), one obtains the relation (\ref{rel}). Now, we calculate the formula
\begin{equation}
\Bigl( \Psi (z+\gamma _j,\bar{z}+\gamma _j) \Bigr) \Bigl( V'_j \Bigr) \Bigl( \Psi ^{-1} (z,\bar{z}) \Bigr), \label{x}
\end{equation}
where $j,k=1,\cdots,n$. By using the relation (\ref{rel}), the formula (\ref{x}) turns out to be
\begin{align*}
&\mathrm{exp} \biggl\{ \frac{\mathbf{i}}{r'} ((\bar{\mathcal{A}}-\mathcal{A})z)_j +\frac{\pi \mathbf{i}}{r'}(\bar{\mathcal{A}}-\mathcal{A})_{jj} +\frac{\mathbf{i}}{r} (\{ (T-\bar{T})^{-1} \}^t (\delta -T)^t \mu )_j \\
&-\frac{\mathbf{i}}{r} (\{ (T-\bar{T})^{-1} \}^t (\delta -\bar{T})^t \bar{\mu } )_j \biggr\} V'_j \\
&=\mathrm{exp} \biggl\{ \frac{2\pi }{r'}(Rz)_j +\frac{2\pi ^2}{r'} R_{jj}  +\frac{\mathbf{i}}{r} (\{ (T-\bar{T})^{-1} \}^t (\delta -T)^t \mu )_j \\
&-\frac{\mathbf{i}}{r} (\{ (T-\bar{T})^{-1} \}^t (\delta -\bar{T})^t \bar{\mu } )_j \biggr\} V'_j \\
&=\mathrm{exp} \biggl\{ \frac{\mathbf{i}}{r} (\{ (T-\bar{T})^{-1} \}^t (\delta -T)^t \mu )_j -\frac{\mathbf{i}}{r} (\{ (T-\bar{T})^{-1} \}^t (\delta -\bar{T})^t \bar{\mu } )_j \biggr\} V'_j \\
&\times \mathrm{exp} \biggl\{ \frac{1}{r'}\mathcal{R}(z,\gamma _j)+\frac{1}{2r'}\mathcal{R}(\gamma _j,\gamma _j) \biggr\}.
\end{align*}
In particular,
\begin{equation*}
\frac{\mathbf{i}}{r} (\{ (T-\bar{T})^{-1} \}^t (\delta -T)^t \mu )_j -\frac{\mathbf{i}}{r} (\{ (T-\bar{T})^{-1} \}^t (\delta -\bar{T})^t \bar{\mu } )_j
\end{equation*}
in this formula is a purely imaginary number, and this fact indicates
\begin{equation*}
\mathrm{exp} \biggl\{ \frac{\mathbf{i}}{r} (\{ (T-\bar{T})^{-1} \}^t (\delta -T)^t \mu )_j -\frac{\mathbf{i}}{r} (\{ (T-\bar{T})^{-1} \}^t (\delta -\bar{T})^t \bar{\mu } )_j \biggr\} V'_j \in U(r').
\end{equation*}
Moreover, for simplicity, we set
\begin{equation*}
U(\gamma_j):=\mathrm{exp} \biggl\{ \frac{\mathbf{i}}{r} (\{ (T-\bar{T})^{-1} \}^t (\delta -T)^t \mu )_j -\frac{\mathbf{i}}{r} (\{ (T-\bar{T})^{-1} \}^t (\delta -\bar{T})^t \bar{\mu } )_j \biggr\} V'_j.
\end{equation*}
Similarly, we also calculate the formula
\begin{equation}
\Bigl( \Psi (z+\gamma '_k,\bar{z}+\bar{\gamma '_k}) \Bigr) \Bigl( e^{-\frac{\mathbf{i}}{r}a_k (\delta -\bar{T})(T-\bar{T})^{-1} z +\frac{\mathbf{i}}{r}a_k (\delta -T)(T-\bar{T})^{-1} \bar{z} } \Bigr) \Bigl( \Psi ^{-1}(z,\bar{z}) \Bigr), \label{y}
\end{equation}
where $j,k=1,\cdots,n$. By a direct calculation, we see that the formula (\ref{y}) turns out to be
\begin{align*}
&\mathrm{exp} \biggl\{ \frac{2\pi }{r'} (\bar{T}^t R z)_k +\frac{\pi \mathbf{i}}{r'} (T^t \bar{\mathcal{A}} T)_{kk} +\frac{\pi \mathbf{i}}{r'} (\bar{T}^t \mathcal{A} \bar{T} )_{kk} -\frac{2\pi \mathbf{i}}{r'} (T^t \mathcal{A} \bar{T} )_{kk} \\
&+\frac{\mathbf{i}}{r}(\bar{T}^t \{ (T-\bar{T})^{-1} \}^t (\delta -T)^t \mu )_k -\frac{\mathbf{i}}{r} (T^t \{ (T-\bar{T})^{-1} \}^t (\delta -\bar{T})^t \bar{\mu })_k \biggr\} U'_k \\
&=\mathrm{exp} \biggl\{ \frac{2\pi }{r'} (\bar{T}^t R z)_k +\frac{\pi \mathbf{i}}{r'}(T^t (\bar{\mathcal{A}}-\mathcal{A}) \bar{T})_{kk} +\frac{\pi \mathbf{i}}{r'}(T^t \bar{\mathcal{A}} (T-\bar{T}))_{kk} -\frac{\pi \mathbf{i}}{r'}((T-\bar{T})^t \mathcal{A} \bar{T} )_{kk} \\
&+\frac{\mathbf{i}}{r}(\bar{T}^t \{ (T-\bar{T})^{-1} \}^t (\delta -T)^t \mu )_k -\frac{\mathbf{i}}{r} (T^t \{ (T-\bar{T})^{-1} \}^t (\delta -\bar{T})^t \bar{\mu })_k \biggr\} U'_k \\
&=\mathrm{exp} \biggl\{ \frac{2\pi }{r'} (\bar{T}^t R z)_k +\frac{2\pi ^2}{r'}(T^t R \bar{T})_{kk} +\frac{\pi \mathbf{i}}{r'}(T^t \bar{\mathcal{A}} (T-\bar{T}))_{kk} -\frac{\pi \mathbf{i}}{r'}((T-\bar{T})^t \mathcal{A} \bar{T} )_{kk} \\
&+\frac{\mathbf{i}}{r}(\bar{T}^t \{ (T-\bar{T})^{-1} \}^t (\delta -T)^t \mu )_k -\frac{\mathbf{i}}{r} (T^t \{ (T-\bar{T})^{-1} \}^t (\delta -\bar{T})^t \bar{\mu })_k \biggr\} U'_k \\
&=\mathrm{exp} \biggl\{ \frac{\pi \mathbf{i}}{r'}(T^t \bar{\mathcal{A}} (T-\bar{T}))_{kk} -\frac{\pi \mathbf{i}}{r'}((T-\bar{T})^t \mathcal{A} \bar{T} )_{kk} \\
&+\frac{\mathbf{i}}{r}(\bar{T}^t \{ (T-\bar{T})^{-1} \}^t (\delta -T)^t \mu )_k -\frac{\mathbf{i}}{r} (T^t \{ (T-\bar{T})^{-1} \}^t (\delta -\bar{T})^t \bar{\mu })_k \biggr\} U'_k \\
&\times \mathrm{exp} \biggl\{ \frac{1}{r'}\mathcal{R}(z,\gamma '_k)+\frac{1}{2r'} \mathcal{R}(\gamma '_k,\gamma '_k) \biggr\}.
\end{align*}
Note that the second equality follows from the relation (\ref{rel}). In particular,
\begin{align*}
&\frac{\pi \mathbf{i}}{r'}(T^t \bar{\mathcal{A}} (T-\bar{T}))_{kk} -\frac{\pi \mathbf{i}}{r'}((T-\bar{T})^t \mathcal{A} \bar{T} )_{kk} \\
&+\frac{\mathbf{i}}{r}(\bar{T}^t \{ (T-\bar{T})^{-1} \}^t (\delta -T)^t \mu )_k -\frac{\mathbf{i}}{r} (T^t \{ (T-\bar{T})^{-1} \}^t (\delta -\bar{T})^t \bar{\mu })_k
\end{align*}
is a purely imaginary number, and this fact indicates
\begin{align*}
&\mathrm{exp} \biggl\{ \frac{\pi \mathbf{i}}{r'}(T^t \bar{\mathcal{A}} (T-\bar{T}))_{kk} -\frac{\pi \mathbf{i}}{r'}((T-\bar{T})^t \mathcal{A} \bar{T} )_{kk} \\
&+\frac{\mathbf{i}}{r}(\bar{T}^t \{ (T-\bar{T})^{-1} \}^t (\delta -T)^t \mu )_k -\frac{\mathbf{i}}{r} (T^t \{ (T-\bar{T})^{-1} \}^t (\delta -\bar{T})^t \bar{\mu })_k \biggr\} U'_k \in U(r').
\end{align*}
Moreover, for simplicity, we set
\begin{align*}
U(\gamma'_k):=&\mathrm{exp} \biggl\{ \frac{\pi \mathbf{i}}{r'}(T^t \bar{\mathcal{A}} (T-\bar{T}))_{kk} -\frac{\pi \mathbf{i}}{r'}((T-\bar{T})^t \mathcal{A} \bar{T} )_{kk} \\
&+\frac{\mathbf{i}}{r}(\bar{T}^t \{ (T-\bar{T})^{-1} \}^t (\delta -T)^t \mu )_k -\frac{\mathbf{i}}{r} (T^t \{ (T-\bar{T})^{-1} \}^t (\delta -\bar{T})^t \bar{\mu })_k \biggr\} U'_k.
\end{align*}
Here, we remark that the matrices $U(\gamma_j)$, $U(\gamma'_k)$ $(j,k=1,\cdots, n)$ satisfy the relations (\ref{cocycle1}), (\ref{cocycle2}) and (\ref{cocycle3}) if and only if the matrices $V'_j$, $U'_k$ $(j,k=1,\cdots, n)$ satisfy the cocycle condition
\begin{equation*}
V'_j V'_k =V'_k V'_j,\ U'_j U'_k =\zeta ^{(A^t \delta )_{jk}-(A^t \delta )_{kj}}U'_k U'_j,\ \zeta ^{-a_{jk}}U'_k V'_j =V'_j U'_k 
\end{equation*}
of $\varphi^*E_{(r,A,\mu, \mathcal{U})}^{T'}$. This completes the proof.
\end{proof}
Finally, we check that the set consisting of projectively flat bundles $\varphi^*E_{(r,A,\mu, \mathcal{U})}^{T'}$ differs from the set consisting of holomorphic vector bundles $E_{(r,A,\mu,\mathcal{U})}^T$. Now, we define a set $\mathscr{E}_{\delta}$ by
\begin{equation*}
\mathscr{E}_{\delta}:=\Bigl\{ \varphi^*E_{(r,A,\mu, \mathcal{U})}^{T'} \ | \ AT'=(AT')^t \Leftrightarrow A^t (-T+\delta)=(A^t (-T+\delta))^t \Bigr\}.
\end{equation*}
Also, we define a set $\mathscr{E}_{\mathrm{SYZ}}$ \footnote{As stated in section 4, the definition of holomorphic vector bundles $E_{(r,A,\mu,\mathcal{U})}^{T'}$ is based on the SYZ construction (SYZ transform), and holomorphic vector bundles $E_{(r,A,\mu,\mathcal{U})}^T$ are defined as an analogue of those holomorphic vector bundles $E_{(r,A,\mu,\mathcal{U})}^{T'}$. Hence, we denoted the set $\left\{ E_{(r,A,\mu, \mathcal{U})}^T \ | \ AT=(AT)^t \right\}$ by $\mathscr{E}_{\mathrm{SYZ}}$ in the above.} by
\begin{equation*}
\mathscr{E}_{\mathrm{SYZ}}:=\left\{ E_{(r,A,\mu, \mathcal{U})}^T \ | \ AT=(AT)^t \right\}.
\end{equation*}
Then, in general, 
\begin{equation*}
\mathscr{E}_{\delta} \not= \mathscr{E}_{\mathrm{SYZ}} 
\end{equation*}
holds as sets. For example, we consider the case
\begin{equation*}
T=\left( \begin{array}{ccc} \mathbf{i} & 1 \\ -1 & \mathbf{i} \end{array} \right), \ \delta:=\left( \begin{array}{ccc} 0 & 0 \\ 0 & 1 \end{array} \right),
\end{equation*}
namely,
\begin{equation*}
T'=(-T+\delta)^{-1}=\left( \begin{array}{ccc} \mathbf{i}+1 & \mathbf{i} \\ -\mathbf{i} & 1 \end{array} \right),
\end{equation*}
and define
\begin{equation*}
A_1:=\left( \begin{array}{ccc} 0 & 1 \\ 1 & 1 \end{array} \right), \ A_2:=\left( \begin{array}{ccc} 1 & 1 \\ 1 & -1 \end{array} \right).
\end{equation*}
In this setting, although $A_1$ satisfies the relation $A_1T'=(A_1T')^t$, it does not satisfy the relation $A_1T=(A_1T)^t$, i.e., $A_1T \not=(A_1T)^t$. Similarly, although $A_2$ satisfies the relation $A_2T=(A_2T)^t$, it does not satisfy the relation $A_2T'=(A_2T')^t$, i.e., $A_2T' \not=(A_2T')^t$. This fact indicates
\begin{equation*}
\varphi^*E_{(1,A_1,0,\mathcal{U}_1)}^{T'} \not\in \mathscr{E}_{\mathrm{SYZ}}, \ E_{(1,A_2,0,\mathcal{U}_2)}^T \not\in \mathscr{E}_{\delta}.
\end{equation*}
Here, note that $\mathcal{U}_1$ and $\mathcal{U}_2$ are defined by using the data $(1,A_1)$ and $(1,A_2)$, respectively. Hence, we can conclude that
\begin{equation*}
\mathscr{E}_{\delta} \not= \mathscr{E}_{\mathrm{SYZ}}
\end{equation*}
holds in this case. Generally, when we discuss the homological mirror symmetry on a mirror pair obtained by using the SYZ construction, we employ holomorphic vector bundles which is obtained by using the SYZ transform in the complex geometry side. Thus, the above result $\mathscr{E}_{\delta} \not= \mathscr{E}_{\mathrm{SYZ}}$ implies that we can discuss the homological mirror symmetry, also by using a class of projectively flat bundles which is different from the class of such holomorphic vector bundles.

\section*{Acknowledgment}
I would like to thank Hiroshige Kajiura for various advices in writing this paper. I am also grateful to referees for useful comments. This work was supported by Grant-in-Aid for JSPS Research Fellow 18J10909.

\end{document}